\documentclass[amsmath,secnumarabic,floatfix,amssymb,nofootinbib,nobibnotes,letterpaper,11pt,tightenlines]{revtex4}

\usepackage{times}
\usepackage{appendix}
\usepackage{geometry}
\usepackage{amssymb}
\usepackage{latexsym, amsmath, amscd,amsthm}
\usepackage{graphicx}
\usepackage[percent]{overpic}
\usepackage{units}
\usepackage{hyperref}
\usepackage{diagrams}
\PassOptionsToPackage{caption=false}{subfig}
\usepackage[lofdepth]{subfig}
\usepackage{booktabs}
\usepackage{mathrsfs}
\usepackage{multirow}
\usepackage{clrscode}

\newtheorem{theorem}{Theorem}

\newtheorem{lemma}[theorem]{Lemma}
\newtheorem{proposition}[theorem]{Proposition}
\newtheorem{corollary}[theorem]{Corollary}

\theoremstyle{definition}
\newtheorem{definition}[theorem]{Definition}

\let\mgp=\marginpar \marginparwidth18mm \marginparsep1mm
\def\marginpar#1{\mgp{\raggedright\tiny #1}}

\let\lbl=\label
\def\label#1{\lbl{#1}\ifinner\else\marginpar{\ref{#1} #1}\ignorespaces\fi}

\begin{document}
\title[]{A de Finetti-style Result for Polygons Drawn from the Symmetric Measure}
\author{Michael Berglund}
\altaffiliation{Piedmont College, Mathematics Department, Demorest GA}
\noaffiliation

\begin{abstract} 
There is a natural intuition that, given a large $n$, the distributions of small segments of a randomly sampled polygonal chain and those of a randomly sampled closed polygonal chain (drawn from the subspace measure of course), should be very similar. We show that this is the case for the symmetric measure on polygon spaces, and provide explicit bounds on the total variation between these two distributions.\end{abstract}
\date{\today}
\maketitle

\section{Introduction} % (fold)
\label{sec:introduction}
Let us begin by defining the symmetric measure for polygons, introduced in \cite{Can1}. This measure is defined as a pushforward, so we will first define the spaces and maps involved. Set $\operatorname{Arm}_d(n)$ to be the moduli space of $n$-edged polygonal chains in $\mathbb{R}^d$ up to translations and dilations, and set $\text{Pol}_d(n)$ to be the subspace consisting of the closed polygons.

Next, set $S:\mathbb{C}^n\rightarrow\mathbb{C}^n$ to be the map which squares each coordinate, and set $H:\mathbb{H}^n\rightarrow\mathbb{H}^n$ to be the map from the $n$-dimensional module over the division ring of quaternions which applies the Hopf map ($\mathbf{q}\mapsto \mathbf{\overline{q}iq}$) to each coordinate. It is important to note that the Hopf map sends any quaternion to a purely imaginary quaternion, in such as way that it produces the fiber bundle $S^1\rightarrow \mathbb{H} \xrightarrow{\operatorname{Hopf}} \mathbb{R}^3$. 

We then note that by considering each polygonal chain as an ordered list of edge vectors, $\operatorname{Arm}_2(n)$ can be identified with the subspace $\{\vec{z}\in\mathbb{C}^n:\sum_{i=1}^n \vert z_i \vert = 2\}$, and that $\operatorname{Arm}_3(n)$ can be identified with the subspace $\{\vec{x}\in(\mathbb{R}^3)^n:\sum_{i=1}^n \| x_i \| = 2\}$\footnote{The total perimeter could be any fixed positive number, but we choose 2 for convenience.}.

In this way, we can view $S$ as a map from $S^{2n-1}(\sqrt 2)\rightarrow \operatorname{Arm}_2(n)$ and $H$ as a map from $S^{4n-1}(\sqrt{2})\rightarrow \operatorname{Arm}_3(n)$. Finally, the symmetric measure on $\operatorname{Arm}_d(n)$ is the pushforward of the Haar measure on the appropriate sphere.

Under the embedding of $\{\vec{a},\vec{b}\}\mapsto \vec{a}+\mathbf{i}\vec{b}$, it is not hard to show that the Stiefel manifold $V_2(\mathbb{R}^n)$ is precisely the preimage of $\operatorname{Pol}_2(n)$. Likewise, under the embedding of $\{\vec{a},\vec{b}\}\mapsto \vec{a}+\mathbf{j}\vec{b}$, we see that $V_2(\mathbb{C}^n)$ is the preimage of $\operatorname{Pol}_3(n)$. Finally, observe that under these embeddings the subspace measure agrees with the pushforward of the Haar measure. Finally, we explicitly mention the following fact: the Haar measure on $S^k$ is invariant under the permutations of the coordinates, and so we find that the symmetric measure for polygons is invariant under permutations of the edge vectors. More in depth detail of this construction can be found in \cite{Can1}. 
% section introduction (end)

\section{The Planar Case}
\label{sec:total-curvature}
We may sample an $n$-edge polygon from the symmetric measure on $\operatorname{Pol}_2(n)$ by applying 
the map $S$ to a 2-frame sampled from the Haar measure on $V_2(\mathbb{R}^n)$. This 2-frame may in turn 
may be obtained by sampling a matrix from the Haar measure on $O(n)$ and taking the first two columns. Likewise, 
we may sample an $n$-edge arm from the symmetric measure on $\operatorname{Arm}_2(n)$ by applying $S$ to a point sampled 
from the spherical measure on $S^{2n-1}(\sqrt{2})$. Since these utilize the same map, if we are interested in the distribution of the first few edges of polygons sampled from the symmetric measure, we need only focus on the distributions of the first few coordinates of 
$\{\vec{v}=(x_1,y_1,x_2,y_2,\dots,x_n,y_n)\in\mathbb{R}^{2n}:\|v\|^2=2\}$ under the embeddings of 
$V_2(\mathbb{R}^n)$ and $S^{2n-1}(\sqrt{2})$ into $\mathbb{R}^{2n}$.

\begin{theorem}[From \cite{Diaconis01}]\label{ortho-df} Suppose that $Z$ is the $r\times s$ upper block of a random matrix $U$ 
which is uniform on $O(n)$, implying that it has mean 0 and covariance of $\frac{1}{n} I_r \otimes I_s$. Let 
$X$ be an $rs$ multivariate normal distribution with the same mean and covariance. Then, provided that 

\noindent
$r+s+2<n$, the total variation distance between the law of $Z$ and the law of $X$ is bounded by 
$B(r,s;n)=2\left( \left(1-\frac{r+s+2}{n}\right)^{-c} -1 \right)$, where $c=\frac{t^2}{2}$ and $t=$min$(r,s)$.
\end{theorem}

Here, $A \otimes B$ is the Kronecker product of $A$ and $B$, given by:

\begin{definition}Where $A=(a_{i,j})$ is an $m\times n$ matrix and $B=(b_{i,j})$ is a $p\times q$ matrix 
we define the Kronecker product $A\otimes B$ to be the $mp\times nq$ matrix, given in block form as 

\begin{center}
$\begin{bmatrix}
a_{1,1}B & \dots & a_{1,n}B \\
\vdots & \ddots & \vdots \\
a_{m,1}B & \dots & a_{m,n}B\end{bmatrix} $

\end{center}
\end{definition}

\begin{theorem}[From \cite{Diaconis02}.]\label{sphere-df} Let $Q_{n,r,k}$ be the law of $\left(\xi_1,\dots,\xi_k\right)$ 
when $\left(\xi_1,\dots,\xi_k,\xi_{k+1},\dots,\xi_n\right)$ is uniformly distributed over the surface of the sphere 
$\displaystyle\left\{\xi:\sum_{i=1}^n \xi_i^2 = r^2\right\}$. Let $P_\sigma^k$ be the law of $\sigma \zeta_1,\dots,\sigma\zeta_k$ 
where the $\zeta$ are independent standard normals. Then the total variation distance between $Q_{n,r,k}$ and $P_{r/\sqrt{n}}^k$ 
is bounded by $2\frac{k+3}{n-k-3}$, for $1\leq k\leq n-4$. 
\end{theorem}

The polygonal chains we wish to sample are being drawn from precisely the distributions that these theorems are concerned with. Moreover, they are supplying an upper bound on much they differ from a normal distribution. We will 
now use this to bound the total variation distance between the distribution of small collections of edges in high-dimensional 
polygons sampled from the respective symmetric measures on $Pol_2(n)$ and $\operatorname{Arm}_2(n)$.

\begin{theorem}\label{plane-B-bound} Let $P(k,n)$ be the law of the first $k$-edged segment of a random $n$-edged 
closed polygon sampled under the symmetric measure on $\operatorname{Pol}_2(n)$, and $A(k,n)$ be the law of the first $k$-edged 
segment of a random $n$-edged arm sampled under the symmetric measure on $\operatorname{Arm}_2(n)$. If $1\leq k\leq n-2$, 
then we have that the total variation between $P(k,n)$ and $A(k,n)$ is bounded above by 
$\mathscr{B}_2(k,n)=2\left(\frac{2k+3}{2n-2k-3}+\frac{(2n-k-4)(k+4)}{(n-k-4)^2}\right)$.
\end{theorem}

\begin{proof} First, notice that $P(k,n)$ is given by the law of $Z$, the $k\times 2$ upper left block of a 
matrix sampled uniformly on $O(n)$. Likewise, $A(k,n)$ would be the law of the first $2k$ coordinates of a 
point sampled from the sphere $\displaystyle\left\{\xi:\sum_{i=1}^{2n} \xi_i^2 = (\sqrt{2})^2\right\}$. To 
use Theorems~\ref{ortho-df} and ~\ref{sphere-df}, we will need $X$, the $2k$ multivariate normal 
distribution with mean 0 and covariance matrix $\frac{1}{n} I_{2k}$, and $P_{\sqrt{2}/\sqrt{2n}}^{2k}$, 
the law of $\frac{1}{\sqrt{n}}\zeta_1,\dots,\frac{1}{\sqrt{n}}\zeta_{2k}$ with the $\zeta_i$ independent 
standard normals. Here we find that this $P_{1/\sqrt{n}}^{2k}$ is a multivariate distribution with mean 0 
and covariance given by 

\noindent
$\left(\frac{1}{\sqrt{n}}I_{2k}\right)\left(\frac{1}{\sqrt{n}} I_{2k}\right)^\intercal$ = $\frac{1}{n} I_{2k}$, 
so we see that $X$ and $P_{1/\sqrt{n}}^{2k}$ are multivariate normal distributions with the same mean and 
covariance. This distribution is often denoted by 
$\displaystyle N\left(\mathbf{0}, \frac{1}{n} I_{2k}\right)$. Since total variation is a norm on 
measures, it satisfies the triangle inequality. We then find that:
\begin{align}
\|P(k,n)-A(k,n)\| 
&\leq \|Z-X\| + \|P_{1/\sqrt{n}}^{2k} - Q_{n,\sqrt{2},2k}\| \\
&\leq 2\left( \left(1-\frac{k+4}{n} \right)^{-2} - 1 \right) + 2\left(\frac{2k+3}{2n-2k-3}\right) \\
%&=2\left( \left(\frac{n-k-4}{n} \right)^{-2} - 1 \right) + 2\left(\frac{2k+3}{2n-2k-3}\right)\\
%&=2\left(\frac{n^2}{(n-k-4)^2} - \frac{n^2-2n(k+4)+(k+4)^2}{(n-k-4)^2} + \frac{2k+3}{2n-2k-3}\right)\\
%&=2\left( \frac{2n(k+4)-(k+4)^2}{(n-k-4)^2} +\frac{2k+3}{2n-2k-3}\right)\\
&=2\left( \frac{(2n-k-4)(k+4)}{(n-k-4)^2} + \frac{2k+3}{2n-2k-3}\right).
\end{align} 
\end{proof}

\begin{proposition}\label{Plane-B-Asymp} $\mathscr{B}_2(k,n)$ is asymptotic to $\frac{6k+19}{n}$. 
\end{proposition}
%\begin{proof} Combining this expression, and using the notation of $o(n^p)$ to represent a quantity, 
%$q(n)$, that satisfies the limiting condition $\displaystyle \lim_{n\rightarrow\infty}\frac{q(n)}{n^p}=0$, we have:
%\vspace{-.25cm}
%\begin{align*}
%\mathscr{B}_2(k,n)&=2\left( \frac{(2n-k-4)(k+4)}{(n-k-4)^2} + \frac{2k+3}{2n-2k-3}\right)\\
%&=2\left( \frac{(2n-k-4)(k+4)(2n-2k-3)}{(n-k-4)^2(2n-2k-3)} + \frac{(2k+3)(n-k-4)^2}{(n-k-4)^2(2n-2k-3)}\right)\\
%&=2\left(\frac{4(k+4)n^2+(2k+3)n^2+o(n^2)}{(n-k-4)^2(2n-2k-3)}\right) \\
%&=2\left(\frac{(6k+19)n^2+o(n^2)}{2n^3+o(n^3)}\right) \\
%&=\frac{6k+19}{n}+o(n^{-1}).\end{align*}
%
%\end{proof}

Before moving on, let us take a moment to discuss this bound in more detail: It 
should be pointed out that $\mathscr{B}_2(k,n)$ is decreasing to this asymptotic. Specifically, for any $k<\frac{n}{2}$, 
we will have $\mathscr{B}_2(k,n)>\frac{6k+19}{n}$. Nonetheless, if we wanted to let $k$ grow with $n$, 
then as long as $k$ is $o(n)$ [for example $k=\lambda n^p$ for any $p\in[0,1)$ and $\lambda\in(0,\infty)$], 
we have that $\displaystyle\lim_{n\rightarrow\infty} \mathscr{B}_2(k,n)=0$. Likewise, if we write $k=\alpha n$, we see that:
\begin{align}
\lim_{n\rightarrow\infty} \mathscr{B}_2(\alpha n,n) &=\lim_{n\rightarrow\infty} 2\left( \frac{(2n-\alpha n-4)(\alpha n+4)}{(n-\alpha n-4)^2} + \frac{2\alpha n+3}{2n-2\alpha n-3}\right)\\
%&=\lim_{n\rightarrow\infty} 2\left( \frac{((2-\alpha)n-4)(\alpha n+4)}{((1-\alpha)n-4)^2} + \frac{2\alpha n+3}{2(1-\alpha)n-3}\right)\\
%&=\lim_{n\rightarrow\infty} 2\left( \frac{(2-\alpha)\alpha n^2-\alpha n + 4(2-\alpha n)-16}{(1-\alpha)^2n^2-8(1-\alpha)n+16} + \frac{2\alpha n+3}{2(1-\alpha)n-3}\right)\\
&=2\left( \frac{(2-\alpha)\alpha}{(1-\alpha)^2} + \frac{\alpha}{1-\alpha}\right)\\
%&=2\alpha \left( \frac{2-\alpha}{(1-\alpha)^2} + \frac{1-\alpha}{(1-\alpha)^2}\right)\\
&=\frac{2\alpha(3-2\alpha)}{(1-\alpha)^2}.\end{align}

Setting up this quadratic inequality, we quickly find that for $\alpha > \frac{4-\sqrt{11}}{5} \simeq 0.136675$, this limit, 
$\frac{2\alpha(3-2\alpha)}{(1-\alpha)^2}$, is greater than 1. We point this out, as this tells us that as long as, in the 
limit, $k< 13\%$ of $n$, we can glean some information about the distribution of $k$-edged segments in 
$\text{Pol}_2(n)$ by virtue of our knowledge of the distribution of $k$-edged segments in $\operatorname{Arm}_2(n)$. 
%To 
%get a better sense of this bound, consider the following table of values of $\mathscr{B}_2(k,n):$
%
%\begin{center}
%\begin{tabular}{cc|c|c|c|c|l}
%\cline{3-6}
%& & \multicolumn{4}{ c| }{$\mathbf{n}$} \\ \cline{3-6}
%& & $300$ & $500$ & $1,000$ & $10,000$ \\ \cline{1-6}
%\multicolumn{1}{ |c| }{\multirow{4}{*}{$\mathbf{k}$} } &
%\multicolumn{1}{ |c| }{2} & 0.106074 & 0.0629767 & 0.0312423 & 0.00310241 & \\ \cline{2-6}
%\multicolumn{1}{ |c }{} &
%\multicolumn{1}{ |c| }{3} & 0.127162 & 0.0753618 & 0.0373375 & 0.00370335 & \\ \cline{2-6}
%\multicolumn{1}{ |c }{\multirow{2}{*}{} } &
%\multicolumn{1}{ |c| }{10} & 0.280319 & 0.163969 & 0.0804659 & 0.00791443 & \\ \cline{2-6}
%\multicolumn{1}{ |c }{} &
%\multicolumn{1}{ |c| }{20} &0.517348 & 0.296629& 0.143515 & 0.0139439 & \\ \cline{1-6}
%\end{tabular}
%\end{center}
%~
%\vspace{10cm}
%~

\begin{definition} Call a function $f\colon\operatorname{Arm}_d(n)\rightarrow \mathbb{R}$ a $k$-edged locally defined 
function if 
$f([\vec{e_1},\vec{e_2},\dots,\vec{e_k},\vec{u_1},\dots,\vec{u_{n-k}}])=f([\vec{e_1},\vec{e_2},\dots,\vec{e_k},\vec{v_1},\dots,\vec{v_{n-k}}])$ 
for all $\vec{e_i},\vec{u_j},\vec{v_j}\in\mathbb{R}^d,$ 

\noindent
$ i\in\{1,2,\dots,k\}, j\in\{1,2,\dots,n-k\}$.
\end{definition}

\begin{theorem}\label{main-planar} Let $f$ be an essentially bounded, $k$-edged 
locally defined function. Then the expectation of $f$ over $\operatorname{Pol}_2(n)$ may be approximated 
by the expectation of $f$ over $\operatorname{Arm}_2(n)$ to within $M \mathscr{B}_2(k,n)$, where $M$ is 
a bound for $f$ almost everywhere.
\end{theorem}

\begin{proof} The expectation of $f$ over $\operatorname{Arm}_2(n)$, $E_{\operatorname{Arm}_2(n)}(f)$, is given by 
$\int_{\operatorname{Arm}_2(n)}f \mu_A$, where $\mu_A$ is the symmetric measure on $\operatorname{Arm}_2(n)$. Likewise, 
the expectation of $f$ over $\text{Pol}_2(n)$, $E_{\text{Pol}_2(n)}(f)$, is given by $\int_{\text{Pol}_2(n)}f \nu_P$, 
where $\nu_P$ is the symmetric measure on $\operatorname{Arm}_2(n)$. Since $f$ is a 

\noindent$k$-edged locally defined 
function, we may integrate out the last $n-k$ edges to obtain 

\noindent$E_{\operatorname{Arm}_2(n)}=\int_{\mathscr{A}_2(k)} f(q) \mu_n^k$, where $\mu_n^k$ is the law of the first 
$k$ edges of a polygon sampled from the symmetric measure on $\operatorname{Arm}_2(n)$. Similarly, we can see 
that $E_{\text{Pol}_2(n)}=\int_{\mathscr{A}_2(k)} f(q) \nu_n^k$, where $\nu_n^k$ is the law of the first 
$k$ edges of a polygon sampled from the symmetric measure on $\text{Pol}_2(n)$. We may therefore write:

%%We can therefore write $E_{\operatorname{Arm}_2(n)}=\int_{S^{2n-1}} (f\circ P) \sigma^{2n-1}$, where $\sigma^{2n-1}$ is the normalized spherical measure on $S^{2n-1}(\sqrt{2})$. Since $f$ is a $k$-edged locally defined function, we may freely integrate out the remaining $2n-2k$ coordinates to obtain:
%%
%%$$E_{\operatorname{Arm}_2(n)} = \underbrace{\int_0^{\sqrt{2}} \dots \int_0^{\sqrt{2}}}_{2k} f(\operatorname{Hopf}(s_1\mathbf{1}+s_2\mathbf{j}),\dots,\operatorname{Hopf}(s_{2k-1}\mathbf{1}+s_{2k}\mathbf{j}))g(s_1,\dots,s_{2k})\mathrm{d}s_1\dots\mathrm{d}s_{2k},$$
%%
%%where $g$ is the probability distribution function for the $2k$-coordinates of a point sampled uniformly on the sphere $S^{2n-1}(\sqrt{2})$. Similarly, we can write $E_{Pol_2(n)}(f)$ as $

%Let $g$ be the function that takes an ordered pair of vectors in $\mathbb{R}^{n}$ to the value 
%of $q$ determined by the first $k$ edges of the polygon determined by Hopf of the pair. Then we 
%have that the expectation of $q$, $E_{pol}(q)$ is the result of integrating $f$ against the masure 
%$\mu=LP(k,n)$ on the space of configurations of $k$ edges. Likewise, we have the measure 
%$\nu=LA(k,n)$. As such, we can write:
%\vspace{-.25cm}
\begin{align}
\vert E_{Pol_2(n)}(f) -E_{\operatorname{Arm}_2(n)}(f)\vert &= \vert \int f \nu_n^k -\int f\mu_n^k\vert \\
&= \vert \int f (\nu_n^k - \mu_n^k) \vert \\
&\leq \|f\|_\infty\|\nu_n^k-\mu_n^k\|_{TV} \\
&\leq \|f\|_\infty\mathscr{B}_2(k,n). \end{align}
\end{proof}

\begin{corollary}\label{main-planar-cor} Let $f$ be an essentially bounded, $k$-edged locally 
defined function. Let $E_p(n)$ stand for the expectation of $f$ over 
$\text{Pol}_2(n)$, and $E_a(n)$ stand for the expectation of $f$ over $\operatorname{Arm}_2(n)$. Further, let 
$\widetilde{E_p}(n)$ stand for the expectation of the sum 
of $f$ over a polygon in $Pol_2(n)$, by which we mean the expectation of the quantity 
$\displaystyle \sum_{i=1}^n f(e_{1+i},e_{2+i},\dots,e_{n+i})$, where the indices are taken modulo $n$.
Likewise, let $\widetilde{E_a}(n)$ stand for the expectation 
of the sum of $f$ over a polygon in $\operatorname{Arm}_2(n)$. 
Provided that $\displaystyle\lim_{n\rightarrow\infty} n E_a(n) = \infty$, then 
$\displaystyle\lim_{n\rightarrow\infty}\frac{E_p(n)}{E_a(n)}=1$ and 
$\displaystyle\lim_{n\rightarrow\infty}\frac{\widetilde{E_p}(n)}{\widetilde{E_a}(n)}=1$.
\end{corollary}

\begin{proof} From Theorem~$\ref{main-planar}$, we see that 
$E_a(n)-M\mathscr{B}_2(k,n)\leq E_p(n)\leq E_a(n)+M\mathscr{B}_2(k,n)$. 
Dividing through by $E_a(n)$, this becomes 
$1-M\frac{\mathscr{B}_2(k,n)}{E_a(n)} \leq \frac{E_p(n)}{E_a(n)}\leq 1+M\frac{\mathscr{B}_2(k,n)}{E_a(n)}$. 
From Proposition~$\ref{Plane-B-Asymp}$, we see that $\mathscr{B}_2(k,n)$ is asymptotic 
to $\frac{6k+19}{n}$. This, in addition to 
our assumption on $\displaystyle \lim_{n\rightarrow \infty}n E_a(n)$, tells us that 
$\displaystyle \lim_{n\rightarrow\infty} M\frac{\mathscr{B}_2(k,n)}{E_a(n)}=\lim_{n\rightarrow\infty} M\frac{6k+19}{n E_a(n)}=0$. The first result 
then follows from the Squeeze Theorem.

In the second case, note that from the invariance under permutations, we have that 

\noindent$\widetilde{E_p}(n)=n E_p(n)$ and $\widetilde{E_a}(n)=(n-k-1) E_a(n)$. 
This produces the inequality 

\noindent$\widetilde{E_a}(n)-nM\mathscr{B}_2(k,n)\leq \widetilde{E_p}(n)\leq \widetilde{E_a}(n)+n M\mathscr{B}_2(k,n)$, 
which we can divide through to produce: 
$1-M\frac{\mathscr{B}_2(k,n)}{E_a(n)}\left(\frac{n}{n-k-1}\right)\leq \frac{\widetilde{E_p}(n)}{\widetilde{E_p}(n)}\leq 1+ M\frac{\mathscr{B}_2(k,n)}{E_a(n)}\left(\frac{n}{n-k-1}\right)$. 
The result then follows from the Squeeze Theorem and our earlier observation that 
$M\frac{\mathscr{B}_2(k,n)}{E_a(n)}\rightarrow 0$ as 
$n\rightarrow\infty$.
\end{proof}

\section[Curvature]{Curvature}
The total curvature of a planar polygon is defined as the sum of the turning angles, and when we sample under the symmetric 
measure, each turning angle has the same expectation. Combining this with the fact that an 
expectation of a sum is the sum of the expectations (even for highly correlated data), we see that the expectation of total curvature will be 
$n$ times the expectation of a turning angle. Since we know that the edge vectors for a polygonal chain in $\operatorname{Arm}_d(n)$ have direction sampled uniformly from the sphere $S^{d-1}$, it is easy to see that the turning angle has expected value of $\frac{\pi}{2}$. Using Theorem~$\ref{main-planar}$, we see that 
$\vert E_{Pol_2(n)}(\theta) - \frac{\pi}{2} \vert \leq \pi \mathscr{B}_2(2,n)$. Moreover, we 
know that a closed polygon will have a higher expected turning angle than a polygonal arm. As 
such, we see that we can bound the expectation of the total curvature of a closed planar polygon as:

$$ 0 \leq E_{pol}(\kappa) - n\frac{\pi}{2} \leq 2n\pi \left( \frac{7}{2n-7} + \frac{12(n-3)}{(n-6)^2} \right)$$

Of particular interest, we see from taking the limit of this inequality, that the expectation of total 
curvature of planar polygons lies between $n\frac{\pi}{2} $ and $31\pi+n\frac{\pi}{2}+O\left(n^{-1}\right)$. 
Of course we already have a trivial upper bound of $n\pi$, but the bound we show is better, provided that $n>69$.

Even though it has already been shown in \cite{Can2} that $\frac{E_{Pol_2(n)}(\kappa)}{E_{\operatorname{Arm}_2(n)}(\kappa)}\rightarrow 1$, 
our corollary here shows this not to be an artifact of total curvature, but of the proximity in distribution 
between pairs of edges in open polygonal chains and pairs of edges in closed polygonal chains. Let us now 
look at the variance of total curvature.

\begin{proposition}\label{adapt-prop} The variance of total curvature of a random polygon sampled under the 
symmetric measure on $Pol_2(n)$ is bounded by 
$$M= \pi^2\left(n\mathscr{B}_2(2,n) + 2n \mathscr{B}_2(3,n)+ (n^2-3n)\mathscr{B}_2(4,n)\right) - n^2 \left(\pi\epsilon_n +\epsilon_n^2\right),$$ 
where $e_n=E_{Pol_2(n)}[\theta_1]-\frac{\pi}{2}$ is surplus of the expectation of the turning 
angle of a polygon over $Pol_2(n)$ over $\frac{\pi}{2}$.
\end{proposition}

\begin{corollary}\label{adapt-cor} The variance of total curvature of a random polygon sampled under the 
symmetric measure on $Pol_2(n)$ is bounded above by $(n \pi)^2\mathscr{B}_2(4,n)\simeq 43n\pi^2.$
\end{corollary}

\begin{proof}[Proof of Proposition]
We know that the covariance of a pair can be computed as 

\noindent$\operatorname{Cov}(\theta_i,\theta_j)=E[(\theta_i-E[\theta_i])(\theta_j-E[\theta_j])]=E[\theta_i\theta_j]-E[\theta_i]E[\theta_j]$. 
We already have established the bounds that $\frac{\pi}{2}\leq E(\theta_1)\leq\frac{\pi}{2}+\pi\mathscr{B}_2(2,n)$. 
For convenience, let $t_n=E(\theta_1)$ and define $\epsilon_n=t_n-\frac{\pi}{2}$, so that $\epsilon_n>0$ and 
$\epsilon_n\rightarrow 0$ as $n\rightarrow \infty$. 

We may partition the pairs $(\theta_i,\theta_j)$ into three categories: (1) 
$j\equiv i \bmod{n}$, 
(2) $j\equiv i\pm 1\bmod{n}$, and (3) $j\equiv i\pm k\bmod{n}$ for 
$1<k<\frac{n}{2}$. By the symmetry of the measure, the covariance of any pair will be the 
same as the covariance of any other pair from the same category. More, we see that in each 
category, the covariance is the integral of an essentially bounded function determined by a set 
of consecutive edges (a pair of edges in category 1, a triple of edges in category 2, and quadruple of edges in category 3).

Recall then that the variance of a sum is equal to the sum of the covariance of the pairs 
%(corollary 8 on page 64 of \cite{basics})
. This tells us that the variance of total 
curvature may be partitioned into the sum:
\begin{align}
Var\left(\sum_{i=1}^n\theta_i\right)&=\left(\sum_{i=1}^n \operatorname{Cov}(\theta_i,\theta_i)\right) +2\left(\sum_{i=1}^{n} \operatorname{Cov}(\theta_i,\theta_{i+1})\right) +2\left(\sum_{i=1}^{n-2}\sum_{k=2}^{n-i} \operatorname{Cov}(\theta_i,\theta_{i+k})\right) \\
%&=\left(\sum_{i=1}^n \operatorname{Cov}(\theta_1,\theta_1)\right) +2\left(\sum_{i=1}^{n} \operatorname{Cov}(\theta_1,\theta_{2})\right) +2\left(\sum_{i=1}^{n-2}\sum_{k=2}^{n-i} \operatorname{Cov}(\theta_1,\theta_3)\right) \\
&=n \operatorname{Cov}(\theta_1,\theta_1) + 2n \operatorname{Cov}(\theta_1,\theta_{2}) +n(n-3) \operatorname{Cov}(\theta_1,\theta_3). \end{align}

By choosing to compute $\operatorname{Cov}(\theta_i,\theta_j)=E[\theta_i\theta_j]-E[\theta_i]E[\theta_j]$, 
and recalling our definition that $t_n=E[\theta_i]$, we may express this as:
\begin{align}
\operatorname{Var}\left(\sum_{i=1}^n\theta_i\right)&=n (E[\theta_1^2]-t_n^2)+2n (E[\theta_1\theta_2]-t_n^2)+(n^2-3n)(E[\theta_1\theta_3]-t_n^2)\\
&= n E[\theta_1^2] + 2n E[\theta_1\theta_2] + (n^2-3n)E[\theta_1 \theta_3] - (n t_n)^2.\end{align}

We can compute $E[\theta_1^2]$ as the integral of a scale-invariant function determined 
by a pair of edges that is essentially bounded by $\pi^2$. This means that we may use 
Theorem~$\ref{main-planar}$ to conclude that 
$\vert E_{Pol_2(n)}[\theta_1^2] - E_{\operatorname{Arm}_2(n)}[\theta_1^2]\vert \leq \pi^2 \mathscr{B}_2(2,n)$. 
A simple calculation shows that $E_{\operatorname{Arm}_2(n)}[\theta_1^2]=\frac{\pi^2}{4}$, so we have that 
$E_{Pol_2(n)}[\theta_1^2]\leq \pi^2 \left(\mathscr{B}_2(2,n)+\frac{1}{4}\right)$.

Likewise, $E[\theta_1\theta_2]$ and $E[\theta_1\theta_3]$ are computed as the integral of 
scale-invariant functions determined by three and four edges respectively. The independence of 
edge directions in $\operatorname{Arm}_2(n)$ tells us that 
$E_{\operatorname{Arm}_2(n)}[\theta_1\theta_2]=E_{\operatorname{Arm}_2(n)}[\theta_1]E_{\operatorname{Arm}_2(n)}[\theta_2]=\frac{\pi^2}{4}$. 
So we see that $E_{Pol_2(n)}[\theta_1]E_{Pol_2(n)}[\theta_3]\leq \pi^2\left(\mathscr{B}_2(3,n)+\frac{1}{4}\right)$ 
and 

\noindent$E_{Pol_2(n)}[\theta_1]E_{Pol_2(n)}[\theta_2]\leq \pi^2\left(\mathscr{B}_2(4,n)+\frac{1}{4}\right)$.

This allows us to place an upper bound on the variance of total curvature for $Pol_2(n)$ as follows:
\begin{align}
\operatorname{Var}\left(\sum_{i=1}^n\theta_i\right)&= n E[\theta_1^2] + 2n E[\theta_1\theta_2] + (n^2-3n)E[\theta_1 \theta_3] - (n t_n)^2\\
&\leq n \pi^2 \mathscr{B}_2(2,n) + 2n \pi^2\mathscr{B}_2(3,n)+ (n^2-3n)\pi^2\mathscr{B}_2(4,n)+\frac{n^2\pi^2}{4} - n^2 (\frac{\pi}{2}+\epsilon_n)^2\\
&\leq \pi^2\left(n\mathscr{B}_2(2,n) + 2n \mathscr{B}_2(3,n)+ (n^2-3n)\mathscr{B}_2(4,n)\right) - n^2 \left(\pi\epsilon_n +\epsilon_n^2\right)\\
&\leq \pi^2 n^2\mathscr{B}_2(4,n)\end{align}

Next, we claim that, as one would naturally suspect, $\mathscr{B}_2(k,n)$ is increasing in $k$. Recall that 
$\mathscr{B}_2(k,n)=2\left(\frac{(2n-k-4)(k+4)}{(n-k-4)^2}+\frac{2k+3}{2n-2k-3}\right)$. The second summand 
is clearly increasing, with $k$, as the denominator is decreasing while the numerator is increasing. The first summand 
likewise has a decreasing denominator, and the numerator, $(2n - (k+4))(k+4)$ is quadratic in $k$ with negative concavity. 
Since the critical point of this quadratic occurs at $k=n-4$ (which is also the largest $k$ for which the bound holds), we see that 
the numerator of the first summand is also increasing.

We can now establish a larger bound by replacing $\mathscr{B}_2(2,n)$ and $\mathscr{B}_2(3,n)$ 
with $\mathscr{B}_2(3,n)$. We then obtain an even larger bound by ignoring the 
$-n(\pi\epsilon_n+\epsilon_n^2)$. 
\end{proof}

Chebyshev's inequality tells us that the probability of a polygon having total curvature $\kappa$ inside the 
range of $[n t_n - \lambda\sqrt{\operatorname{Var}}, nt_n + \lambda\sqrt{\operatorname{Var}}]$ is at least $1-\frac{1}{\lambda^2}$. By 
the above corollary, we can extend this to a slightly larger, but easier to work with interval by replacing 
$\sqrt{\operatorname{Var}}$ with $n\pi\sqrt{\mathscr{B}_2(4,n)}$. This interval is then 
$[n\left( \frac{\pi}{2}+\epsilon_n - \lambda \pi \sqrt{\mathscr{B}_2(4,n)}\right) , n \left( \frac{\pi}{2} +\epsilon_n+ \lambda \sqrt{\mathscr{B}_2(4,n)}\right)]$. We know from \cite{Can2} that $\epsilon_n$ is asymptotic to $\frac{2}{n\pi}$. We also know from Proposition~$\ref{Plane-B-Asymp}$ that $\mathscr{B}_2(4,n)$ is asymptotic to, and less than, $\frac{43}{n}$, so we see that, asymptotically, $\epsilon_n < \mathscr{B}_2(4,n)$. Since we only obtain useful information when $\lambda > 1$, this interval may be 
augmented to $\left[n\pi\left(\frac{\pi}{2}-\frac{7\lambda}{\sqrt{n}}\right), n\pi\left(\frac{\pi}{2}+2\frac{7\lambda}{\sqrt{n}}\right)\right]$. 
Notice here, that the length of the interval is $21\pi\sqrt{n}$. So the length of this 
interval is growing at a rate of $O(\sqrt{n})$.

Of course we already have the trivial bounds that all planar polygons in $Pol_2(n)$ will 
have total curvature between $2\pi$ and $n \pi$, so let us first check that these bounds 
are better than that. By setting $\lambda=\sqrt{2}$, we can say that at most $\frac{1}{2}$ 
of the polygons in $Pol_2(n)$, lie outside our given bounds, and that the lower will be larger 
than $2\pi$ when $n\geq 48$, while the upper will be smaller than $n \pi$ when $n\geq 159$. 
Past those marks, our bounds from variance become more useful than the trivial bounds. 

Before moving on, we would like to point out that this analysis was intended merely as an example, and is in fact adaptable for any essentially bounded $k$-edged locally defined function\footnote{For further example, the interested reader will be able to verify that the arguments leading to Proposition~$\ref{adapt-prop}$ 
and Corollary~$\ref{adapt-cor}$ could be slightly adjusted to say that the variance of the sum of $f$ over all $n$ runs of 
consecutive $k$-edges in a polygon is bounded by the quantity $(n M)^2\mathscr{B}_2(2k,n)$, where $M$ is a bound for $f$ almost everywhere.}.

\section{The Spatial Case} %(fold)
\label{sec:spatial-case}

To produce the spatial analogue to Theorem~\ref{plane-B-bound} and Theorem~\ref{main-planar}, we will need an analogue to Theorem~\ref{ortho-df}. In \cite{Diaconis01}, such an adaptation is left to the interested reader, as are many of the tools needed along the way. So as not to interupt the flow, these details have been placed in the Appendix and pick up with these new theorems:

\begin{theorem}\label{my-df}Let $Z$ be the upper left $r\times s$ block of a random matrix $U$ which is uniform on $U(n)$, 
so that it has density given by Theorem~$\ref{dist}$. Further, we have that $EZ=O\in\mathscr{M}_{r,s}(\mathbb{C})$ and 

\noindent$\operatorname{Cov}(Z)=n^{-1}I_r\otimes I_s$, so we shall take $X$ to be a random matrix with the 
$r\times s$ complex multivariate Gaussian distribution with the same mean and covariace. Then, provided that $r+s+2<n$, 
the variation distance between $\mathscr{L}(Z)$ and $\mathscr{L}(X)$ is 
bounded above by $B(r,s;n):=2\left(\left(1-\frac{r+s}{n}\right)^{-t^2}-1\right)$, where $t=\min(r,s)$.
\end{theorem}

\begin{theorem}\label{main-space} Let $f$ be an essentially bounded $k$-edged locally defined function. 
Then the expectation of $f$ over $\operatorname{Pol}_3(n)$ may be approximated by 
the expectation of $f$ over $\operatorname{Arm}_3(n)$ to within $M \mathscr{B}_3(k,n)$, where $M$ is a bound for $f$ 
almost everywhere, and

$$\mathscr{B}_3(k,n):=B(k,2;n) = 2 \left(\frac{4k+3}{4n-4k-3}+\frac{n^4}{(n-k-2)^4}-1\right)$$
\end{theorem}

\begin{theorem}\label{main-space-int} Let $f$ be an essentially bounded, $k$-edged 
locally defined function. Then the expectation of $f$ over $\operatorname{Pol}_2(n)$ may be approximated 
by the expectation of $f$ over $\operatorname{Arm}_2(n)$ to within $M \mathscr{B}_2(k,n)$, where $M$ is 
a bound for $f$ almost everywhere.
\end{theorem}

\begin{corollary}\label{space-cor} Let $q$ be an essentially bounded, locally measured quantity of a 
polygonal chain. Let $E_p(n)$ stand for the expectation of $q$ over $\operatorname{Pol}_3(n)$, and $E_a(n)$ stand 
for the expectation of $q$ over $\operatorname{Arm}_3(n)$. If $n E_a(n) \rightarrow \infty$, then $\frac{E_p(n)}{E_a(n)}\rightarrow 1$.

Moreover, the expectation of the sum of $q$ over the polygon, $\widetilde{E_p}(n)$ and the 
expectation of the sum of $q$ over the polygon, $\overline{E_p}(n)$ also satisfy 
$\frac{\widetilde{E_p}(n)}{\widetilde{E_a}(n)}\rightarrow 1$ and $\frac{\overline{E_p}(n)}{\overline{E_a}(n)}\rightarrow 1$.
\end{corollary}

\begin{proof} The proofs of these mirror those given in for Theorem~$\ref{plane-B-bound}$, Theorem~$\ref{main-planar}$ and 
Corollary~$\ref{main-planar-cor}$, where we replace $V_2(\mathbb{R}^n)$ with 
$V_2(\mathbb{C}^n)$, $S^{2n}$ with $S^{4n}$, and our variation bounds from the shared, 
close-proximity multivariate Gaussian, come from Theorem~$\ref{my-df}$ and Theorem~$\ref{sphere-df}$ respectively.
\end{proof}

Looking at this bound, we see that, as with the planar case, it is limiting to 0 at a rate of $O(n)$. Specifically, for any fixed $k$, we have  $\displaystyle\lim_{n\rightarrow \infty} n \mathscr{B}_3(k,n)=10k+\frac{35}{2}$, and we 
additionally have again find that $\mathscr{B}_3(k,n)$ is strictly less than the asymptotic, provided that the bound is useful 
($\mathscr{B}_2(k,n)$ is greater than 2 for $k>\frac{n}{5}$). We again find that when $k=o(n^p)$ with $0<p<1$, this is 
limiting to 0, and that when $k=\alpha n$, $\mathscr{B}_3(\alpha n,n)$ 
is limiting to an understandable quantity, this time $\frac{2}{1-\alpha}+\frac{2}{(1-\alpha)^4}-2$, which is greater than 1 for $\alpha>0.08235533$. 
In other words, provided that the the number of edges $k$ is less than $8\%$ of $n$, we are able to say that 
the distributions of $k$-edged segments coming from $\operatorname{Arm}_3(n)$ are close enough in total variation to those coming 
from $\operatorname{Pol}_3(n)$ to hope to apply our theorems.

\section{Torsion} % (fold)
\label{sec:new-measure}

In \cite{Can2}, we can see that the integral to find the expected total 
curvature with respect to the symmetric measure on $\operatorname{Pol}_3(n)$ is explicitly computed as $E(\kappa;\operatorname{Pol}_3(n),\nu_P)=\frac{\pi}{2}n+\frac{\pi}{4}\frac{2n}{2n-3}$. Let us now then attempt to
solve the problem of finding bounds on the total torsion.

\begin{definition} For a polygon in $\mathbb{R}^3$, we define the torsion angle (sometimes called the 
dihedral angle) at an edge $e_i$ by the following procedure: Let $p_i$ be the plane which is normal to 
$e_i$ at $v_i$. Project edges $e_{i-1}$ and ${e_{i+1}}$ to $p_i$ along $v_i$ to get a 2-edge planar 
polygonal arm in $p_i$ with middle vertex $v_i$. The torsion angle is then defined as the angle between 
these edges, with the convention that we take its value in the range $(-\pi,\pi]$.
\end{definition}

\begin{proposition} The distribution of the torsion angle for arms is the same as the distribution of 
$\pi-\theta$, where $\theta$ is the polar angle in the spherical coordinates of the edges.
\end{proposition}
\begin{proof}
Write $e_i=(r_i,\theta_i,\phi_i)$ in spherical coordinates. Rotate the configuration so that $e_i$ is on the 
$z$-axis. If we then further rotate so that $e_{i-1}$ has no $y$-component, we can see that the projections 
to $p_i$ (the $xy$-plane) form a planar 2-edge arm that runs along the negative $x$-axis, then turns to form 
an edge given in polar coordinates as $(\tilde{r}_{i+1}, \theta_{i+1})$. As such, the torsion angle of the 
rotated configuration will be given by $\pi-\theta_{i+1}$. Since the distribution of arms is invariant under 
the $SO(3)$ action\footnote{This follows from the $U(n)$-invariance by taking an appropriate block diagonal matix.} on $\mathbb{R}^3$, the result follows.
\end{proof}

\begin{proposition} The expectation of the torsion angle of a polygonal arm sampled under the symmetric 
measure on $\operatorname{Arm}_3(n)$ is $0$.
\end{proposition}
\begin{proof} Similar to how we found the expectation of curvature, since we know that the symmetric 
measure is expressible as a product measure on $\mathbb{R}^n\times\left(S^2\right)^n$, with the spherical 
measure on the individual copies of $S^2$, we see that the distribution of the polar angles will be uniform on 
$[0,2\pi)$, so the expectation of $\pi-\theta$ will be 0.
\end{proof}

For polygons, we have an integral even more imposing than the one for planar polygon's curvature. So this 
is an excellent opportunity to use the total variation 
bound. Using Theorem~$\ref{main-space}$, we see that $\vert E(\tau_i) - 0\vert \leq \pi \mathscr{B}_3(3,n)$. 
This give us bounds on total torsion of $ -n \pi \mathscr{B}_3(3,n) \leq E(\tau) \leq n \pi \mathscr{B}_3(3,n)$. 
This is limiting to the range of $[-55.5\pi,55.5\pi]$. However, unlike total curvature, the expectation of total 
torsion over $\operatorname{Arm}_3(n)$ is 0 for any $n$. As such, $\displaystyle\lim_{n\rightarrow\infty} n E_{\operatorname{Arm}_3(n)}(\tau) = 0$, 
so we may not apply Corollary~$\ref{space-cor}$. Nonetheless, perhaps we may hope to glean useful information 
by considering the variance.

\begin{proposition} Where $\tau_i$ is the torsion angle at edge $e_i$ of a polygon sampled under the symmetric 
measure on $\operatorname{Arm}_3(n)$, we have that $\operatorname{Cov}(\tau_i,\tau_j)=\delta_{i,j}\frac{1}{3}\pi^2$.
\end{proposition}
\begin{proof}
We can easily see from the independence of directions and our earlier description of the dihedral angle, that the 
covariance of any distinct pair of dihedral angles will be 0. So let us focus on $\operatorname{Cov}(\tau_i,\tau_i)$ where we have:

\begin{align}
\operatorname{Cov}(\tau_i,\tau_i)&=\operatorname{Var}(\tau_i)\\
&=\int_{\operatorname{Arm}_3(n)}\!\left(\tau_i-0\right)^2\,\mathrm{d}\sigma\\
&=\frac{1}{2\pi}\int_0^{2\pi}\! (\pi-\theta_i)^2\,\mathrm{d}\theta_i\\
&=\left. \frac{-1}{6\pi}(\pi-\theta_i)^3\right\vert_{\theta_i=0}^{\theta_i=2\pi}\\
&=\frac{1}{6\pi}(2\pi^3)\\
&=\frac{1}{3}\pi^3.\end{align}
\end{proof}

For an explicit example, notice that this means the variance of total torsion for an open polygonal arm is 
$\frac{n}{3}\pi^2$, which pairs with Chebyshev's inequality to tell us that we should expect less than 
one-third of all open polygonal arms to have total torsion with absolute value greater than $\pi\sqrt{n}$. 
%Given this, we sampled 100,000 polygons from $Arm_3(100)$ and found that 91,626 had total torsion 
%less than $\pi\sqrt{100}$. Further, they had a mean of $-0.030508$ and a variance of only 
%$1.257$ larger than $\frac{100}{3}\pi^2$.

\begin{proposition}\label{tor-bound} Where $\tau_i$ is the torsion angle at edge $e_i$ of a closed polygon sampled under 
the symmetric measure on $Pol_3(n)$, we have that the variance of total torsion 
$\tau=\sum_{i=1}^n \tau_i$ is bounded above by $\frac{n}{3}\pi^2 + n^2\pi^2\mathscr{B}_3(6,n)$. 
\end{proposition}

\begin{proof} The $U(n)$ invariance will again suggest that we should partition the pairs of torsion angles 
into: (A) $(\tau_i,\tau_i)$, (B) $(\tau_i,\tau_{i\pm 1})$, (C) $(\tau_i,\tau_{i\pm 2})$ and (D) all others. Within 
these categories, those in (A) have covariance equal to $\operatorname{Cov}_p(\tau_1,\tau_1)$, those in (B) will match 
$\operatorname{Cov}_p(\tau_1,\tau_2)$, those in (C) will match $\operatorname{Cov}_p(\tau_1,\tau_3)$ and those in (D) will match 
$\operatorname{Cov}_p(\tau_1,\tau_4)$. Breaking the variance apart, we have:
\begin{align}
\operatorname{Var}\left(\sum_{i=1}^n \tau_i\right) &= \sum_{i=1}^n\sum_{j=1}^n \operatorname{Cov}_p(\tau_i,\tau_j) \\
&= n \operatorname{Cov}_p(\tau_1,\tau_1) + 2n \operatorname{Cov}_p(\tau_1,\tau_2) + 2n \operatorname{Cov}_p(\tau_1,\tau_3) + (n^2-5n) \operatorname{Cov}_p(\tau_1,\tau_4)\\
&= n E_p(\tau_1^2) + 2n E_p(\tau_1 \tau_2) + 2n E_p(\tau_1 \tau_3) + (n^2-5n) E_p(\tau_1\tau_4)-n^2 E[\tau_1]^2.\end{align}

Here, we see that both $E_p(\tau_1)$ and $E_p(\tau_1^2)$ are obtained as the integral of an essentially 
bounded $3$-edge locally determined function, and similarly we need 4 edges for $E_p(\tau_1 \tau_2)$, 5 
for $E_p(\tau_1 \tau_3)$ and 6 for $E_p(\tau_1 \tau_4)$. We have seen that, over arms, $E(\tau_i,\tau_j)=\delta_{i,j}\frac{1}{3}\pi^2$, 
so we may bound $E_p(\tau_1^2)\leq \frac{1}{3}\pi^2 + \pi^2 \mathscr{B}_3(6,n)$, and 
$E_p(\tau_i \tau_j) \leq 0 + \mathscr{B}_3(6,n)$ for $i<j$, by using the fact that, for fixed $n$, $\mathscr{B}_3(k,n)$ 
is an increasing function of $k$. To see this fact, recall that we have 
$\mathscr{B}_3(k,n)=2 \left(\frac{4k+3}{4n-4k-3}+\frac{n^4}{(n-k-3)^4}-1\right),$ written as the sum of three quantities, only the 
first two of which depend on $k$. In the sum, the first summand has an increasing numerator and decreasing denominator as $k$ increases, while 
the second has constant numerator and decreasing denominator. This shows us that $B_3(k,n)$ is increasing in $k$ (within its domain). 
This leaves us with:
\begin{align}
\operatorname{Var}\left(\sum_{i=1}^n \tau_i\right) &= n E_p(\tau_1^2) + 2n E_p(\tau_1 \tau_2) + 2n E_p(\tau_1 \tau_3) + (n^2-5n) E_p(\tau_1\tau_4)-n^2 E[\tau_1]^2\\
&\leq \frac{n}{3}\pi^2 + n^2\pi^2\mathscr{B}_3(6,n)-n^2E_p[\tau_1^2]\\
&\leq \frac{n}{3}\pi^2 + n^2\pi^2\mathscr{B}_3(6,n)\end{align}
\end{proof}

This bound is asymptotically bounded by $86\pi^2 n$. At the moment, all we can say using only Chebyshev's Inequality 
and our earlier observation about the bounds on expected total 
torsion, is that, for large $n$, we expect that at least $\left(1-\frac{1}{\lambda^2}\right)100\%$ of 
polygons in $Pol_3(n)$ have total torsion in the range of 
$\pm \pi \left(55.5 + \lambda \sqrt{86n}\right)$. While not seemingly very impressive, it is better than the trivial bounds on total torsion in the case of $\lambda=\sqrt{2}$ for $n>272$. That being said, we unfortunately find that $\pi^2n^2\mathscr{B}_3(6,n)>55.5\pi$ for $n>8$, which makes this bound simply too high to be used profitably with 
Chebyshev's Inequality.
%once $n>10,350$. 

%As an example, we sampled 100,000 100-edged closed polygons and computed their total torsion. We found, of course, that all 100,000 had total torsion between the range of $[-55.5\pi, 55.5\pi]$. More, they had a sample mean of 1.2629 and a sample variance of 327.96. 

%\subsection{Subsection} % (fold)
%\label{sub:gaussians}

\section{Future Directions}
From looking at the distribution function for sub-arms (\cite{my-paper} and \cite{Can2}), it is clear that many expectations are significantly easier to explicitly compute for $\operatorname{Arm}_d(n)$ opposed to $\operatorname{Pol}_d(n)$. Additionally, in some of these theorems (e.g. Theorem~\ref{my-df}), the sharpness of the bounds listed is unknown. In particular, a number of numerical experiments hint that it is the case that there may indeed be room for some improvement, so this is definitely a topic for further investigation.

\section{Acknowledgements}
The author is happy to acknowledge the contributions of the many friends and colleagues who provided numerous helpful discussions on polygons and probability theory, especially Edward Azoff, Jason Cantarella, Harrison Chapman, Tom Needham, and Clayton Shonkwiler.

\begin{appendices}
\section{Technical Proofs}\label{sec:appendix}

Throughout this section we will 
let $\mathscr{L}(\ast)$ denote ``the law of $\ast$," as is the convention in many of the references.

\begin{definition} Given a subspace $M$ of $\mathbb{C}^n$, the compact subgroup $U_n(M)\subset U(n)$ 
is defined by $U_n(M)=\{g\in U(n) \vert g x = x\,\text{for all }x\in M\}$.
\end{definition}

\begin{definition} Since $U_n(M)$ is compact, we may pushforward the Haar measure on $U(n)$ to $U_n(M)$ 
and then normalize this pushforward to produce a measure $\nu_M$ on $U_n(M)$.
\end{definition}

\begin{definition} We say that $U$ is uniform on $U_n(M)$ if it is a 
random element with law $\nu_M$.
\end{definition}

\begin{definition} Let $P$ be the orthogonal projection onto the 
$m$-dimensional subspace $M\subset\mathbb{C}^n$ and 
set $Q=I-P$ to be the orthogonal projection onto $M^\perp$. Let $r$ 
be no larger than $n-m$ and let $\alpha$ be a complex matrix of size 
$r\times n$. Define $A(M,\alpha)=\alpha Q \alpha^\ast$. Further, since $Q$ is 
Hermitian, we see that $A(M,\alpha)$ will be Hermitian. The Spectral Theorem  
then tells us that there exists a unitary matrix $U_A$ and a real diagonal matrix $D$ such that $A(M,\alpha)=U^\ast D U$. 
Since $A(M,\alpha)$ is positive semi-definite, we know that all elements of $D$ are non-negative, so it makes since to 
define the matrix $D^{1/2}$ to be the matrix whose $(i,j)-$entry is the non-negative square root of the $(i,j)-$entry of $D$. We then 
define 
$A^{1/2}(M,\alpha):=U^\ast D^{1/2} U$. In particular, note that 
$$A^{1/2}(M,\alpha)A^{1/2}(M,\alpha)=U^\ast D^{1/2} U U^\ast D^{1/2} U=U^\ast D^{1/2} I D^{1/2} U = U^\ast D U = A.$$
\end{definition}

%\begin{definition} Let $P$ be the orthogonal projection onto the 
%$m$-dimensional subspace $M\subset\mathbb{C}^n$ and 
%set $Q=I-P$ to be the orthogonal projection onto $M^\perp$. Let $r$ 
%and $s$ fixed to be no larger than $n-m$. Then for two given matrices 
%$\alpha: r\times n$ and $\beta: s\times n$, define $A(M,\alpha)=\alpha Q \alpha^\ast$, 
%and $B(M,\beta)=\beta Q \beta ^\ast$. Both $A$ and $B$ are bounded positive semi-definite operators. As such they each have a unique non-negative definite square root, which we will denote $A^{1/2}$ and $B^{1/2}$.
%\end{definition}

\begin{lemma} Fix an $m$-dimensional subspace $M\subset \mathbb{C}^n$, and let $P$ be the 
projection matrix for $M$. Let $U$ be uniformly distributed on $U(n-m)$ and let $Z$ be 
the upper left $r\times s$ corner block of $U$. Let $\alpha$ be an $r\times n$ complex matrix 
and let $\beta$ be an $s\times n$ complex matrix, where $r$ and $s$ are no larger 
than $n-m$. For $A=A(M,\alpha)$, $B=A(M,\beta)$, and the variate $V=\alpha U \beta^\ast$, we have 
$\mathscr{L}(V)=\mathscr{L}(A^{1/2} Z B^{1/2} + \alpha P \beta^\ast).$

\end{lemma}
\begin{proof}
First, notice that for any $m$-dimensional subspace $M$, and any $\Gamma\in U(n)$, 
the subgroup $U_n(\Gamma M) = \{ g \in U(n) \,:\, g x = x\text{ for all }x\in \Gamma M\}$, 
is equal to the subgroup $\Gamma U_n(M) \Gamma^\ast$. To see this, note that 
if $g \in U_n(M)$ and $x\in \Gamma M$, then there is a unique $y \in M$ so that $x=\Gamma y$. Then 
$(\Gamma g \Gamma^\ast) x = (\Gamma g) y = \Gamma y = x$, so $\Gamma g \Gamma^\ast \in U_n(\Gamma M)$. Further, if 
$h\in U_n(\Gamma M)$, and $y\in M$, then $h \Gamma y = \Gamma y$. Multiplying on the left by $\Gamma^\ast$ then shows us that 
$\Gamma^\ast h \Gamma y = y$, so $\Gamma^\ast h \Gamma \in U_n(M)$. This then tells us that 
$h\in \Gamma U_n(M) \Gamma^\ast$. Together, 
these give us the relationships that 

\noindent$\Gamma U_n(M) \Gamma^\ast \subseteq U_n(\Gamma M) \subseteq \Gamma U_n(M)\Gamma^\ast$ as desired.

Next, since $U_n(\Gamma M) = \Gamma U_n(M)\Gamma ^\ast$, it suffices to establish the lemma in the case where 

\noindent$M=M_0=\left\{\vec{z}\in\mathbb{C}^n \,:\, \vec{z}=\begin{bmatrix}
\vec{x} \\
\vec{0} \end{bmatrix},\,\vec{x}\in\mathbb{C}^m\right\} $. 
For $M_0$, it is clear that 

\noindent$U_n(M_0)=\left\{g\in O_n \,:\, g=\begin{bmatrix}
I_m & O \\
O & h \end{bmatrix},\,h\in U(n-m)\right\}$. 
Hence, if $U$ is uniform on $U(n-m)$, then $U_0= \begin{bmatrix}
I_m & O \\
O & U \end{bmatrix}$
is uniform on $U_n(M_0)$. 

%Because $\mathscr{L}(U)=\mathscr{L}(-I_n U)$, we have that $E U=O$, 
%so $E(U_0)= 
Set $P_0=\begin{bmatrix}
I_m & O \\
O & O\end{bmatrix}$, the orthogonal projection onto $M_0$, and set 
$Q_0=I-P_0=\begin{bmatrix} O & O \\ O & I_{n-m}\end{bmatrix}$. We can write $Q_0=C_0 C_0^\ast$, 
where $C_0$ is the $n\times (n-m)$ matrix $\begin{bmatrix}
O \\
I_{n-m}\end{bmatrix} $. Then for any $V=\alpha U_0 \beta^\ast$, we have that 

\begin{align}
V&=\alpha I_nU_0I_n\beta^\ast\\
&=\alpha(P_0+Q_0)U_0(P_0+Q_0)\beta^\ast\\
&=\alpha (P_0U_0+Q_0U_0)(P_0+Q_0) \beta^\ast \\ %+ \alpha P_0 \beta^\ast
&=\alpha (P_0U_0P_0+Q_0U_0P_0+P_0U_0Q_0+Q_0U_0Q_0) \beta^\ast \\%+ \alpha P_0 \beta^\ast
&=\alpha (P_0P_0U_0+Q_0P_0U_0+U_0P_0Q_0+Q_0U_0Q_0) \beta^\ast \\%+ \alpha P_0 \beta^\ast
&=\alpha (P_0+OU_0+U_0O+Q_0U_0Q_0) \beta^\ast \\%+ \alpha P_0 \beta^\ast
&=\alpha (P_0+Q_0U_0Q_0) \beta^\ast \\%+ \alpha P_0 \beta^\ast
&= \alpha P_0 \beta^\ast + \alpha Q_0 U_0 Q_0 \beta^\ast.\end{align}
In 31 we use the identity that $U_0=I_n U_0 I_n$ and in 32 the identity that $I_n=P_0+Q_0$. Lines 33 and 34 follow from the 
distributive property. Line 35 comes from the identity that 

\noindent$P_0 U_0 = U_0 P_0 = P_0$. Line 36 follows from the identity that 
$P_0Q_0=Q_0P_0=O$. We have then that $V=\alpha Q_0 U_0 Q_0 \beta^\ast +\alpha P_0 \beta^\ast$ 
$=\alpha C_0 C_0^\ast U_0 C_0 C_0^\ast \beta^\ast + \alpha P_0 \beta^\ast = \gamma U \delta^\ast + \alpha P_0 \beta^\ast$, 
where $\gamma=\alpha C_0$ and $\delta=\beta C_0$ and we have used the fact that 
$C_0^\ast U_0 C_0=C_0^\ast \begin{bmatrix} O \\ U\end{bmatrix}$
$ = U$. Now notice that we have $A_0=\gamma \gamma^\ast = \alpha Q_0\alpha^\ast=A(M_0,\alpha)$, 
and $B_0=\delta \delta^\ast = \beta Q_0 \beta^\ast=A(M_0,\beta)$. This allows us to write $\gamma$ and $\delta$ in their polar 
decompositions \cite{Gal}, as $\gamma=A_0^{1/2}\begin{bmatrix} I_r & O\end{bmatrix} \psi_1$, and 
$\delta=B_0^{1/2}\begin{bmatrix}I_s & O\end{bmatrix} \psi_2$, where $\psi_1,\psi_2\in U(n-m)$. Recalling that $U$ is 
uniform on $U(n-m)$, and is 
thus sampled from the Haar measure, we see that $\mathscr{L}(U)=\mathscr{L}(\psi_1 U \psi_2^\ast)$ which gives us:

\begin{align}
\mathscr{L}(V) &=\mathscr{L}(\alpha Q_0 U_0 Q_0 \beta^\ast + \alpha P_0 \beta^\ast) \\
&=\mathscr{L}\left(A_0^{1/2} \begin{bmatrix} I_r & O\end{bmatrix}\psi_1 U \psi_2^\ast \begin{bmatrix}
I_s \\
O \end{bmatrix} B_0^{1/2}+\alpha P_0 \beta^\ast \right)\\
&= \mathscr{L}\left((A_0^{1/2}\begin{bmatrix} I_r & O\end{bmatrix}U \begin{bmatrix}
I_s \\
O \end{bmatrix} B_0^{1/2} + \alpha P_0 \beta^\ast\right)\\
&= \mathscr{L}(A_0^{1/2} Z B_0^{1/2} + \alpha P_0 \beta^\ast)\end{align}

Where $Z=\begin{bmatrix} I_r & O\end{bmatrix} U \begin{bmatrix}
I_s \\
O \end{bmatrix}$ is the $r\times s$ upper left block of $U$, as desired.
\end{proof}

We may view the Stiefel manifold $V_q(\mathbb{C}^n)$ as the set of all $n\times q$ complex matrices $A$ 
that satisfy $A^\ast A = I_q$. Further, it is well known that if $\Gamma$ is uniform on $U(n)$ 
then $\Gamma_1=\Gamma\begin{bmatrix} I_q \\ O \end{bmatrix}$, is uniform on $V_q(\mathbb{C}^n)$. 

\begin{definition} For a compact group $G$ acting on a measurable space $\mathscr{Y}$, a function 
$\tau: \mathscr{Y}\rightarrow\mathscr{Z}$ is 
called a maximal invariant function under $G$ if: (1) $\tau(g y) =\tau(y)$ for all $y\in\mathscr{Y}$ and $g\in G$ and (2) for 
any pair of points $y_1,y_2\in\mathscr{Y}$ such that $\tau(y_1)=\tau(y_2)$, 
there exists some $g\in\mathscr{Y}$ such that $g y_1 =y_2$.
\end{definition}

%Next, we will need the following proposition from $\cite{Eaton1}$. 
\begin{proposition}[From \cite{Eaton1}]\label{eaton} Suppose that $G$ is a compact group that 
acts on a measurable space $\mathscr{Y}$. Let 
$\tau: \mathscr{Y}\rightarrow\mathscr{Z}$ be a maximal invariant function, and for $i=1,2$, let $Z_i=\tau(Y_i)$ for two $G$-invariant distributions $P_i=\mathscr{L}(Y_i)$. If $\mathscr{L}(Z_1)=\mathscr{L}(Z_2)$, then $P_1=P_2$.
\end{proposition}

Next, for $q\leq p$, partition $\Gamma_1=\begin{bmatrix} \Delta \\ \Psi \end{bmatrix}$, where 
$\Delta$ is $p\times q$ and $\Psi$ is $(n-p)\times q$. 
Additionally, let $\mathbb{L}_{q,n}$ be the space of all $n\times q$ complex matrices of rank $q$, and note that 
$V_q(\mathbb{C}^n)\subsetneq \mathbb{L}_{q,n}$.

\begin{proposition} Suppose $X\in\mathbb{L}_{q,n}$ has a left $U(n)$-invariant distribution. Let 
$\phi:\mathbb{L}_{q,n}\rightarrow V_q(\mathbb{C}^n)$ satisfy $\phi(g x)= g\phi(x)$ for all $x\in\mathbb{L}_{q,n}$ and $g\in U(n)$, 
which is to say that $\phi$ is an equivariant map. Then $\mathscr{L}(\phi(X))=\mathscr{L}(\Gamma_1)$. In other words, 
the image of any invariant distribution under an equivariant map is the Haar measure on $V_q(\mathbb{C}^n)$.
\end{proposition}

\begin{proof} From the uniqueness of the uniform distribution on $V_q(\mathbb{C}^n)$, it suffices to show that 
$\mathscr{L}(g\phi(X))=\mathscr{L}(\phi(X))$ for $g\in U(n)$.We have from assumption on $\phi$ that 

\noindent$\mathscr{L}(g\phi(X))=\mathscr{L}(\phi(g X))$ and from left $U(n)$-invariance that $\mathscr{L}(\phi(g X))=\mathscr{L}(\phi(X))$
\end{proof}

Notice here that a particular such $\phi$ is given by $\phi(x)=x (x^\ast x)^{-1/2}$, (the unitary matrix of the polar decomposition of 
the matrix $x$, as seen in Lemma 2.1 of \cite{Highham}), as we have that 

\noindent$\phi(g x) = g x((gx)^\ast gx)^{-1/2}=g x(x^\ast g^\ast g x)^{-1/2}=g x(x^\ast x)^{-1/2}$.

\begin{proposition}\label{long-delta}
Let $X\in\mathbb{L}_{q,n}$ and partition it into $X=\begin{bmatrix} Y \\ Z \end{bmatrix}$, $Y:p\times q$, $Z:(n-p)\times q$. 
Then $\mathscr{L}(\Delta)=\mathscr{L}(Y(Y^\ast Y+ Z^\ast Z)^{-1/2})$, where again $\Delta$ is the top $p\times q$ block of $\Gamma_1$. 
\end{proposition}
\begin{proof}
We have then, that 
$X^\ast X = \begin{bmatrix} Y^\ast & Z^\ast\end{bmatrix} \begin{bmatrix} Y \\ Z \end{bmatrix} = Y^\ast Y + Z^\ast Z$, 
so the matrix 

\noindent$Y (Y^\ast Y + Z^\ast Z)^{-1/2}$ is the upper $p\times q$ block of $X(X^\ast X)^{-1/2}$. The result 
then follows from the previous proposition.
\end{proof}

%We can likewise restate this in terms of $\Delta^\ast$ as:
\begin{proposition}\label{delta-transpose} Let $U\in\mathbb{L}_{p,n}$ and partition it into 
$U=\begin{bmatrix} V \\ W \end{bmatrix}$, $V:q\times p$, $W:(n-q)\times p$. 
Then $\mathscr{L}(\Delta^\ast)=\mathscr{L}(V(V^\ast V+W^\ast W)^{-1/2})$ and 
$\mathscr{L}(\Delta)=\mathscr{L}((V^\ast V+W^\ast W)^{-1/2}V^\ast).$
%Let $U\in \mathscr{L}_{p,n}$, and partition it into $U=\begin{bmatrix} V \\ W \end{bmatrix}$, for $V:q\times p$, $W:(n-q)\times p$. Then $\mathscr{L}(\Delta^\ast)=\mathscr{L}(V(V^\ast V+ W^\ast W)^{-1/2})$, so we have that 
%
%\noindent
%$\mathscr{L}(\Delta)=\mathscr{L}((V^\ast V+ W^\ast W)^{-1/2}V^\ast)$.
\end{proposition}

\begin{proof} By mirroring the previous proof, we see that $\mathscr{L}(\Delta^\ast)=\mathscr{L}(V(V^\ast V+W^\ast W)^{-1/2})$. Since 
$V^\ast V + W^\ast W $ is Hermitian, so too is its square root. This tells us that 

\noindent$(V(V^\ast V+W^\ast W)^{-1/2})^\ast=(V^\ast V+W^\ast W)^{-1/2})V^\ast$, so we can conclude that 

\noindent$\mathscr{L}(\Delta)=\mathscr{L}((V^\ast V+W^\ast W)^{-1/2}V^\ast)$.
\end{proof}

We now have the tools needed to find explicitly the density of these distributions. First, we will define 
the densities we will be using:

\begin{definition}[From \cite{Goodman}] For a matrix distribution $Y$, whose $n$ rows are independent and identically distributed 
$p$-variate complex Gaussian random variables with covariance matrix $\Sigma$. Then the distribution of 
$Y^\ast Y = \sum_{k=1}^n Y_i Y_i^\ast$, has the probability density function given by:
$$p_W(A) = \frac{\operatorname{det}(A)^{n-p}}{\pi^{\frac{1}{2}p(p-1)}\Gamma(n)\cdots\Gamma(n-p+1)\operatorname{det}(\Sigma)^n}e^{-\operatorname{tr}(\Sigma^{-1}A)},$$
defined on the set of Hermitian positive semi-definite $p\times p$ matrices. This distribution is known as 
the Complex Wishart distribution and we will denote it as $\mathscr{CW}(p,n,\Sigma)$
\end{definition}

Next, we point out that in \cite{Diaz} matrices with the above distribution are said to have the complex matrix variate 
gamma distribution $\mathscr{CG}_p(n,\Sigma)$. 

\begin{definition}[From \cite{Diaz}] For $A\sim \mathscr{CG}_m(a,I_m)=\mathscr{CW}(m,a,I_m)$ and 
$B\sim \mathscr{CG}_m(b,I_m)=\mathscr{CW}(m,b,I_m)$, define the complex matrix variate beta 
type I distribution as either of 
\begin{align}
(1) & \;U=(A+B)^{-1/2} A((A+B)^{-1/2}))\\
(2) & \;V = A^{1/2}(A+B)^{-1}(A^{1/2}).
\end{align}
Further, the density function of this distribution, denoted as $\mathscr{CBI}_m(U;a,b)$, is given by:
$$p_B(M)= \frac{\mathscr{C}\Gamma_m(a+b)}{\mathscr{C}\Gamma_m(a)\mathscr{C}\Gamma_m(b)}\operatorname{det}(M)^{a-m}\operatorname{det}(I_m-M)^{b-m},$$
defined on the set of $m\times m$ Hermitian positive semi-definite matrices $M$, where $\mathscr{C}\Gamma_m(a)$ stands for 
$\pi^{m(m-1)/2}\prod_{j=1}^m \Gamma(a-j+1)$.
\end{definition}

\begin{proposition}\label{cbi}$\mathscr{L}(\Delta^\ast \Delta)=\mathscr{CBI}_q(p,n-p)$ and 
$\mathscr{L}(\Delta\Delta^\ast)=\mathscr{CBI}_p(q,n-q)$.
Further, $\mathscr{L}(\Delta^\ast\Delta)$ has a density given by:
$$p(\Delta^\ast \Delta) =\frac{\mathscr{C}\Gamma_q(n)}{\mathscr{C}\Gamma_q(p)\mathscr{C}\Gamma_q(n-p)}\operatorname{det}(\Delta^\ast \Delta)^{p-q}\operatorname{det}(I_q-\Delta^\ast \Delta)^{n-p-q}$$

\end{proposition}

\begin{proof} Let $X$ be distributed as $N(0,I_n\otimes I_q)$, and be partitioned as 
$X=\begin{bmatrix} Y\\ Z\end{bmatrix}$.

From our definition of the Complex Wishart distribution,

 \noindent$\mathscr{L}(Y^\ast Y)=\mathscr{CW}(q,p,I_q)=\mathscr{CG}_q(p,I_q)$ and 
$\mathscr{L}(Z^\ast Z)=\mathscr{CW}(q,n-p,I_q)=\mathscr{CG}_q(n-p,I_q)$. Next, we see from Proposition~$\ref{delta-transpose}$ that

\noindent$\mathscr{L}(\Delta \Delta^\ast)=\mathscr{L}(((Y^\ast Y + Z^\ast Z)^{-1/2})Y^\ast Y(Y^\ast Y + Z^\ast Z)^{-1/2})$. Finally, from the 
definition of the complex matrix variate beta type I distribution, since this is in the form 

\noindent$U=(A+B)^{-1/2} A((A+B)^{-1/2}))$ for 
$A=Y^\ast Y \sim \mathscr{CG}_q(p,I_q)$ and $B=Z^\ast Z\sim \mathscr{CG}_q(n-p,I_q)$, we see that 

\noindent$\Delta \Delta^\ast$ has a 
distribution of type $\mathscr{CBI}_p(q,n-q)$. Likewise, we see that $\Delta^\ast \Delta\sim\mathscr{CBI}_q(p,n-p)$ and a density function given by:
\begin{align}%
p(\Delta^\ast \Delta)&= \frac{\mathscr{C}\Gamma_q(p+n-p)}{\mathscr{C}\Gamma_q(p)\mathscr{C}\Gamma_q(n-p)}\operatorname{det}(\Delta^\ast \Delta)^{p-q}\operatorname{det}(I_q-\Delta^\ast \Delta)^{n-p-q}\\
&=\frac{\mathscr{C}\Gamma_q(n)}{\mathscr{C}\Gamma_q(p)\mathscr{C}\Gamma_q(n-p)}\operatorname{det}(\Delta^\ast \Delta)^{p-q}\operatorname{det}(I_q-\Delta^\ast \Delta)^{n-p-q}
\end{align}
%&=\pi^{-q(q-1)/2}\prod_{j=1}^q\left(\frac{\Gamma(n-j+1)}{\Gamma(p-j+1)\Gamma(n-p-j+1)}\right)\operatorname{det}(\Delta^\ast \Delta)^{p-q}\operatorname{det}(I_q-\Delta^\ast \Delta)^{n-p-q}\end{align*}
\end{proof}
%%
%%
%%\begin{proposition}\label{cbi} $\mathscr{L}(\Delta^\ast\Delta)=\mathscr{CBI}_q(p,n-p)=$ 
%%$\frac{1}{c} \lvert \Delta^\ast \Delta \rvert^{p-q}\lvert I_q-\Delta^\ast \Delta\rvert^{n-p-q}$, where 
%%$c=\frac{\pi^{q(q-1)/2} (p)_q (n-p)_q}{(n)_q }$, where we are using the short-hand that $(x)_y$ denotes the product 
%%$\prod_{j=1}^y (x-j)!=\frac{G[x+y+1]}{G[x+1]}$, where $x>y\in \mathbb{N}$, and $G[z]=\prod_{i=1}^{z-2}i!$ is the Barnes' G-function 
%%(5.17 of $\cite{dlmf}$). 
%%\end{proposition}
%%\begin{proof} 
%%This proof is the fifth of the five things I'm doing this --it was put in in a very unpleasant fashion, so I'm mostly just doing formatting things.
%%%Notice now that we have that $Y^\ast Y$ is a complex Wishart distribution $W(I_q,q,p)$ and $Z^\ast Z$ is a complex Wishart distribution $W(I_q,q,n-p)$ (these follow from the definitions see $\cite{SRIVASTAVA}$], %or page 162 of Goodman, why aren't we citing that?
%%%leading to probability dentity functions on the set of $p\times p$ Hermitian matricies of $p_W(x)=[\vert x \vert^{n-p}e^{-tr(I_{q}^{-1}x)}]/ \pi^{(1/2)p(p-1)}\Gamma(n)\dots\Gamma(n-p+1)\vert I_q \vert^n$. Combinging this with $\cite{Diaz}$ (Lemma 2.1---Definition 1), we see that $\mathscr{L}(\Delta^\ast\Delta)=\mathscr{CBI}_q(U;p,n-p)$, and as such has a density function given by $p(U)= c \vert U\vert^{p-q} \vert I_q-U \vert^{n-p-q}$, where
%%%
%%%\noindent
%%%$c=\frac{\pi^{q(q-1)/2}\prod_{j=1}^n \Gamma(p-j+1)\Gamma(n-p-j+1) }{\prod_{j=1}^n \Gamma(n-j+1) }$.
%%\end{proof}

\begin{theorem}[From \cite{SRIVASTAVA}]\label{Sriv} For a complex matrix $M$ of size 
$p\times q$, if the density of $M$ depends only on the matrix $B=M^\ast M$, by a function $f(B)$, then the 
density of $B=M^\ast M$ is given by 
$\displaystyle \frac{f(B)\operatorname{det}(B)^{p-q}\pi^{q(p-(1/2)(q-1))}}{\prod_{j=1}^q \Gamma(p-j+1)}$
\end{theorem}

%\begin{proof} 
%This is the main theorem of \cite{SRIVASTAVA}, so the proof is omitted here.
%\end{proof}

We know have the tools needed to determine the probability density function of $\Delta$:

\begin{theorem}\label{dist} For the the upper $p\times q$ block of $\Gamma_1$, called $\Delta$, the density of $\Delta$ is given by

\noindent$f(\Delta)=c_1\lvert I_q-\Delta^\ast \Delta\rvert^{n-p-q},$ where $c$ is the constant given by 
\begin{align*}
c_1&=\pi^{qp}\prod_{j=1}^q\left(\frac{\Gamma(n-j+1)}{\Gamma(n-p-j+1)}\right).\end{align*}

\end{theorem}

\begin{proof}
It follows from the Proposition~$\ref{cbi}$ that $\Delta^\ast\Delta$ has a density given by 
$\mathscr{CBI}_q(p,n-p)$. First, we have that the distribution of $\Delta$ is invariant under the action of $U(p)$ given by left multiplication, 
$\Delta\rightarrow g \Delta, g\in U(p)$. Second, we have a maximal invariant given by $\tau(\Delta)=\Delta^\ast \Delta$. 
Let $\Psi$ be the random matrix variate 
with density given by~$f$. $U(p)$ acts on $\Psi$, and the density of the maximal invariant $\tau(\Psi)$ is then calculated 
from Theorem~$\ref{Sriv}$ as 
\begin{align}
h(\Psi^\ast\Psi)&=\frac{c_1 \operatorname{det}( I_q - \Psi^\ast \Psi)^{n-p-q}\operatorname{det}( \Psi^\ast \Psi)^{p-q} \pi^{q(p-(1/2)(q-1))}}{\prod_{j=1}^q \Gamma(p-j+1)}\\
&=\pi^{qp}\prod_{j=1}^q\left(\frac{\Gamma(n-j+1)}{\Gamma(n-p-j+1)}\right)\frac{\operatorname{det}( I_q - \Psi^\ast \Psi)^{n-p-q}\operatorname{det}( \Psi^\ast \Psi)^{p-q} \pi^{q(p-(1/2)(q-1))}}{\prod_{j=1}^q \Gamma(p-j+1)}\\
&=\pi^{-q(q-1)/2}\prod_{j=1}^q\frac{\Gamma(n-j+1)}{\Gamma(p-j+1)\Gamma(n-p-j+1)}\operatorname{det}( \Psi^\ast \Psi)^{p-q}\operatorname{det}( I_q - \Psi^\ast \Psi)^{n-p-q},\end{align}
This calculation shows that $\mathscr{L}(\Psi^\ast \Psi)=\mathscr{CBI}_q(p,n-p)$, so we see that 
$\mathscr{L}(\Psi^\ast \Psi)=\mathscr{L}(\Delta^\ast\Delta)$. Since we can see that the distribution of 
$\Psi$ is invariant under the group action of $U(p)$, 
it follows from Proposition~$\ref{eaton}$ that $\mathscr{L}(\Psi)=\mathscr{L}(\Delta)$. Hence, $f$ 
must be the density of $\Delta$.
\end{proof}

\noindent
{\bf Theorem~\ref{my-df}.} {\it Let $Z$ be the upper left $r\times s$ block of a random matrix $U$ which is uniform on $U(n)$, 
so that it has density given by Theorem~$\ref{dist}$. Further, we have that $EZ=O\in\mathscr{M}_{r,s}(\mathbb{C})$ and 

\noindent$\operatorname{Cov}(Z)=n^{-1}I_r\otimes I_s$, so we shall take $X$ to be a random matrix with the 
$r\times s$ complex multivariate Gaussian distribution with the same mean and covariace. Then, provided that $r+s+2<n$, 
the variation distance between $\mathscr{L}(Z)$ and $\mathscr{L}(X)$ is 
bounded above by $B(r,s;n):=2\left(\left(1-\frac{r+s}{n}\right)^{-t^2}-1\right)$, where $t=\min(r,s)$.}

\begin{proof} Setting $\mathscr{L}(X)=P_1$ and $\mathscr{L}(Z)=P_2$, let us start with the 
case of $s\leq r$. The density $f_1$ of $P_1$ is given by 
$f(x)=\frac{1}{\pi^{rs} }e^{-tr(x^\ast x)}$ \cite{Goodman}. The density $f_2$ of $P_2$ is 
given in Theorem~$\ref{dist}$. 
Since these are functions of $x^\ast x$ and $z^\ast z$ respectively, the variation distance is equal 
to the variation distance between the distributions of 
$x^\ast x$ and $z^\ast z$. $x^\ast x$ has, in accordance with the definition above, 
the complex Wishart distribution $\mathscr{CW}\left(s,r,\frac{1}{n}I_s\right)$, and hence a density given by
\begin{align}
f(v)&=\frac{\operatorname{det}(v)^{r-s}}{\pi^{\frac{1}{2}s(s-1)}\Gamma(r)\cdots\Gamma(r-s+1)\operatorname{det}(\frac{1}{n}I_s)^r}e^{-\operatorname{tr}((\frac{1}{n}I_s)^{-1}v)}\\
&=\frac{\operatorname{det}(v)^{r-s}}{\pi^{\frac{1}{2}s(s-1)}\Gamma(r)\cdots\Gamma(r-s+1)n^{-sr}}e^{-n\operatorname{tr}(v)}\\
&=\operatorname{det}(v)^{r-s}e^{-n\operatorname{tr}(v)}\pi^{-\frac{1}{2}s(s-1)}\frac{n^{rs}}{\prod_{j=1}^s \Gamma(r-j+1)},
\end{align}
defined 
on the set of $s\times s$ Hermitian, positive-definite matrices. 
The density of $z^\ast z$ we have seen in Proposition~$\ref{cbi}$ to be given by 
\begin{align}
g(v)&= \frac{\mathscr{C}\Gamma_s(n)}{\mathscr{C}\Gamma_s(r)\mathscr{C}\Gamma_s(n-r)}\operatorname{det}(v)^{r-s}\operatorname{det}(I_s-v)^{n-r-s}\\
&= \frac{\operatorname{det}(v)^{r-s}\operatorname{det}(I_s-v)^{n-r-s} \pi^{\frac{1}{2}s(s-1)}\prod_{j=1}^s\Gamma(n-j+1)}{(\pi^{\frac{1}{2}s(s-1)}\prod_{j=1}^s\Gamma(r-j+1))(\pi^{\frac{1}{2}s(s-1)}\prod_{j=1}^s\Gamma(n-r-j+1))}\\
&=\operatorname{det}(v)^{r-s}\operatorname{det}(I_s-v)^{n-r-s}\pi^{-\frac{1}{2}s(s-1)}\prod_{j=1}^s\frac{\Gamma(n-j+1)}{\Gamma(r-j+1)\Gamma(n-r-j+1)},
\end{align}
defined on the set of matrices with $v$ and $I-v$ positive definite. By an alternate characterization 
of total variation (seen in \cite{Diaconis01}), we see that
the total variation distance is given by 

\noindent$\delta_{r,s,n}:=\int \lvert g(v)-f(v)\rvert dv = 2 \int_E \left(\frac{g(v)}{f(v)}-1\right)f(v)dv$, 
where $E$ is the set of $s\times s$ positive definite matrices on which $g(v)>f(v)$. As we will be using it often, let us now 
simplify the expression $\frac{g(v)}{f(v)}$:
\begin{align}
\frac{g(v)}{f(v)}&=\frac{\operatorname{det}(v)^{r-s}\operatorname{det}(I_s-v)^{n-r-s}\pi^{-\frac{1}{2}s(s-1)}\prod_{j=1}^s\frac{\Gamma(n-j+1)}{\Gamma(r-j+1)\Gamma(n-r-j+1)}}{\operatorname{det}(v)^{r-s}e^{-n\operatorname{tr}(v)}\pi^{-\frac{1}{2}s(s-1)}\frac{n^{rs}}{\prod_{j=1}^s \Gamma(r-j+1)}}\\
&=\frac{\operatorname{det}(I_s-v)^{n-r-s}}{e^{-n\operatorname{tr}(v)}n^{rs}}\prod_{j=1}^s\frac{\Gamma(n-j+1)}{\Gamma(n-r-j+1)}\\
&=\operatorname{det}(I_s-v)^{n-r-s}e^{n\operatorname{tr}(v)} \prod_{j=1}^s\frac{\Gamma(n-j+1)}{n^r\Gamma(n-r-j+1)}
\end{align}

Hence, $\delta_{r,s,n}\leq 2\sup_{v\in E} \left(\frac{g(v)}{f(v)}-1\right)$. Set $M_{r,s,n}:=\sup_{v\in E} \left(\frac{g(v)}{f(v)}-1\right)$, 
so that 
$\delta_{r,s,n}~\leq~2~M_{r,s,n}$. Differentiation shows that the maximum of $\left(\frac{g(v)}{f(v)}-1\right)$ is attained uniquely for 
$v=\frac{r+s}{n}I_s$:

Let us first write $\frac{g(v)}{f(v)}=c \operatorname{det}(I_s-v)^{n-r-s} e^{n\operatorname{tr}(v)}$, with 
$c=\prod_{j=1}^s \frac{\Gamma(n-j+1)}{n^r\Gamma(n-r-j+1)}$ independent of $v$. Next, computing the derivative 
with respect to $v$, we will look at first at the partials from the entries off the diagonal, and secondly at the entries of the diagonal.

Case 1: ($i\neq j$) In this case, we first note that $\frac{\partial}{\partial v_{i,j}} e^{n\operatorname{tr}(v)}=0$, 
as the trace depends only on the diagonal. This tells us that 
$\frac{\partial}{\partial v_{i,j}} \frac{g(v)}{f(v)} = c e^{n\operatorname{tr}(v)}\frac{\partial}{\partial v_{i,j}}(\operatorname{det}(I_s-v))^{n-r-s}$. 
Applying the Power Rule and Chain Rule, we see that 

\noindent$\frac{\partial}{\partial v_{i,j}}(\operatorname{det}(I_s-v))^{n-r-s} = (n-r-s)(\operatorname{det}(I_s-v))^{n-r-s-1}\frac{\partial}{\partial v_{i,j}}\operatorname{det}(I_s-v)$. 
Next, we see from 2.1.1 of \cite{cookbook} that 
$\frac{\partial}{\partial v_{i,j}}\operatorname{det}(I_s-v)=\operatorname{det}(I_s-v)\operatorname{tr}\left((I_s-v)^{-1}\frac{\partial}{\partial v_{i,j}}(I_s-v)\right)$. 
Here, $\frac{\partial}{\partial v_{i,j}}(I_s-v)$ is a matrix whose only non-zero entry the $(i,j)$-entry, which is a -1. Hence, we see that that 
$\operatorname{tr}\left((I_s-v)^{-1}\frac{\partial}{\partial v_{i,j}}(I_s-v)\right)$ is $(j,i)$-entry of $-(I_s-v)^{-1}$. Recall that for an invertible matrix $M$, 
$M^{-1}=\frac{1}{\operatorname{det}(M)}\operatorname{adj}(M)=\frac{1}{\operatorname{det}(M)}\operatorname{C}(M)^\intercal$, 
where $\operatorname{adj}(M)$ is the adjoint matrix and $\operatorname{C}(M)$ is the cofactor matrix (3.1.2 and 3.1.4 of \cite{cookbook}). 
Therefore, we see that the $(j,i)$-entry of $-(I_s-v)^{-1}$ is the $(i,j)$-entry of $\frac{-1}{\operatorname{det}(I_s-v)}\operatorname{C}(I_s-v)$. 
We may then conclude that 

\noindent$\frac{\partial}{\partial v_{i,j}} \frac{g(v)}{f(v)}=-c e^{n\operatorname{tr}(v)}(n-r-s)(\operatorname{det}(I_s-v))^{n-r-s-1}\operatorname{C}(I_s-v)_{\{i,j\}}$. 
We can then see that this will only be zero when $\operatorname{C}(I_s-v)_{\{i,j\}}$ is zero, as the first three terms are all positive, and the determinant term is non-zero 
as $g(v)$ is only defined on the set of matrices with both $v$ and $I_s-v$ positive definite. 

Case 2: ($i=j$). We first apply the Product Rule to see that 
$$\frac{\partial}{\partial v_{i,i}} \frac{g(v)}{f(v)} = c\left( e^{n\operatorname{tr}(v)} \left(\frac{\partial}{\partial v_{i,i}} (\operatorname{det}(I_s-v))^{n-r-s}\right) + (\operatorname{det}(I_s-v))^{n-r-s}\left(\frac{\partial}{\partial v_{i,i}}e^{n\operatorname{tr}(v)} \right)\right).$$
We have already computed the partial derivative of the power of the determinant. In the second term, we see from a quick application of the chain rule that 
$\frac{\partial}{\partial v_{i,i}}e^{n\operatorname{tr}(v)}=n e^{n\operatorname{tr}(v)}$. We may then conclude that 
$$\frac{\partial}{\partial v_{i,i}} \frac{g(v)}{f(v)} = c e^{n\operatorname{tr}(v)}(\operatorname{det}(I_s-v))^{n-r-s-1}\left( -(n-r-s)\operatorname{C}(I_s-v)_{\{i,i\}} + n\operatorname{det}(I_s-v)\right).$$
This will be zero only when $n\operatorname{det}(I_s-v)=(n-r-s)\operatorname{C}(I_s-v)_{\{i,i\}}$.

We have now classified the critical point of $\frac{g(v)}{f(v)}$ to be any matrix $v$ for which the $(i,j)$ cofactor of $(I_s-v)$ is given by the equation 
$\frac{n}{n-r-s}\operatorname{det}(I_s-v)\delta_{i,j}$, where $\delta_{i,j}$ is the Kronecker delta. We have already seen how to express 
the inverse of a matrix in terms of the determinant and the cofactor matrix, so since we know all of the cofactors of $I_s-v$, we know the inverse of $I_s-v$. 
Specifically, $\begin{bmatrix} (I_s-v)^{-1}_{i,j}\end{bmatrix} = \frac{1}{\operatorname{det}(I_s-v)}\begin{bmatrix} \frac{n}{n-r-s}\operatorname{det}(I_s-v) \delta_{j,i} \end{bmatrix}$. 
Observe that the matrix on the right-hand side of the equation is simply the identity matrix scaled by $\frac{n}{n-r-s}$. Inverting both sides produces 
$I_s-v=\frac{n-r-s}{n}I_s$, so we see that $v=\left(1-\frac{n-r-s}{n}\right)I_s = \frac{r+s}{n} I_s$.

%For many, an easier way to think of this is to first identify $\mathbb{C}^{s^2}\cong\mathscr{M}_{s\times s}$ by the map 
%$(z_1,z_2,\dots,z_{s^2})\mapsto \begin{bmatrix} z_1 & z_{a} & z_{a} & \dots & z_{a}\\ z_a & z_{2} & z_{a} & \dots & z_{a}\\ \vdots &\vdots &\vdots &\ddots &\vdots \\ z_a&z_{a}&z_{a}&\dots&z_{s}\end{bmatrix}$. 
%In this way, our critical point corresponds to the point $\widetilde{v}=\frac{r+s}{n}(\underbrace{1,1,\dots,1}_{s},\underbrace{0,0,\dots,0}_{s(s-1)})$, with the 
%gradient taking the form 
%
%\noindent
%$\nabla\widetilde{\left(\frac{g}{f}\right)}=c n e^{n\operatorname{tr}(v)}(\operatorname{det}(I_s-v))^{n-r-s-1}(g_1(v),g_2(v),\dots,g_{s^2}(v))$, 
%where $$g_i(v) = \left\{
% \begin{array}{lr}
% 1 & : i \leq s\\
% \operatorname{C}(I_s-v)_{\{i,j\}} & : i > s
% \end{array}
% \right.$$

Now that we see that this is the only critical point, we will show that it produces a maximum. All of the following properties are given 
in \cite{convex}. First, recall that a critical point of a concave function must be a maximum. Second, note that if 
$\phi(x)$ is convex, then so too are $\alpha \phi(x)$, $\phi(x+t)$, and $\phi(A x)$ for any 
$\alpha>0,t\in\mathbb{R}^M$, and $M\times M$ matrix $A$ and $-\phi(x)$ is concave. Fourth, we know that the sum, product, and 
composition of convex functions are convex. From this last property, we see that a concave function pre-composed with a convex function 
is concave and the product of a convex function and a concave function is concave. Using these properties, it is easy to see that the 
trace of a matrix is convex, as it is the sum of the projections to the the diagonal elements. Likewise, from the fact that 
$\frac{d^2}{d^2 x}e^{\alpha x}=\alpha^2e^{\alpha x}$, we know that $e^{\alpha x}$ is convex, showing that 
$c e^{n\operatorname{tr}(v)}$ is convex. Now, we need only show that $\operatorname{det}(v)$ is concave to show the concavity of 
$\frac{f(v)}{g(v)}-1$, which is given as Example 3.39 of \cite{convex}.

Hence, we know that
\begin{align}
M_{r,s,n}+1&=\frac{g((r+s)n^{-1}I_s)}{f((r+s)n^{-1}I_s)}\\
&=\operatorname{det}(I_s-(r+s)n^{-1}I_s)^{n-r-s}e^{n\operatorname{tr}((r+s)n^{-1}I_s)} \prod_{j=1}^s\frac{\Gamma(n-j+1)}{n^r\Gamma(n-r-j+1)}\\
&=\operatorname{det}\left(\left(1-\frac{r+s}{n}\right)I_s\right)^{n-r-s}e^{ns(\frac{r+s}{n})} \prod_{j=1}^s\frac{\Gamma(n-j+1)}{n^r\Gamma(n-r-j+1)}\\
&=\left(1-\frac{r+s}{n}\right)^{s(n-r-s)}e^{s(r+s)} \prod_{j=1}^s\frac{\Gamma(n-j+1)}{n^r\Gamma(n-r-j+1)}\\
&= \prod_{j=1}^s\left(\frac{\Gamma(n-j+1)}{n^r\Gamma(n-r-j+1)}\left(1-\frac{r+s}{n}\right)^{n-r-s}e^{r+s}\right)
\end{align}
We would now like to write this in terms of logarithms. To do this, we first observe that
\begin{align}
-n\int_0^{t}\!\ln(1-x)\,\mathrm{d}x&=-n\left((x-1)\ln(1-x)-x\right)^{x=t}_{x=0}\\
&=-n((t-1)\ln(1-t)-t)\\
&=nt+(n-nt)\ln(1-t).\end{align}
Setting $t=\frac{r+s}{n}$ gives us 
$\displaystyle -n\int_0^{\frac{r+s}{n}}\!\ln(1-x)\,\mathrm{d}x = (r+s)+(n-r-s)\ln\left(1-\frac{r+s}{n}\right)$. 
Next, set
$$A_j=\ln\left(\frac{\Gamma(n-j+1)}{n^r\Gamma(n-r-j+1)}\right)-n\int_0^{(r+s)/n} \ln(1-x)\mathrm{d}x+\ln\left(1-\frac{r+s}{n}\right),$$ 
we can write $M_{r,s,n}+1=\prod_{j=1}^s e^{A_j}$. Now let us write $A_j$ in a more 
pliable form by noting that
\begin{align}
\ln\left(\frac{\Gamma(n-j+1)}{n^r\Gamma(n-r-j+1)}\right) &= \ln(\Gamma(n-j+1))-\ln(\Gamma(n-r-j+1))-\ln(n^r)\\
&= \left(\sum_{i=1}^{n-j} \ln(i)\right)-\left(\sum_{i=1}^{n-r-j} 
\ln(i)\right)-\left(\sum_{i=1}^{r} \ln(n)\right)\\
&=\left(\sum_{i=n-j-r+1}^{n-j} \ln(i)\right)-\left(\sum_{i=1}^{r} \ln(n)\right)\\
&=\left(\sum_{k=1}^{r} \ln(n-j-k+1)\right)-\left(\sum_{i=1}^{r} \ln(n)\right)\\
&=\sum_{i=1}^r \ln\left(\frac{n-j-i+1}{n}\right)\\
&=\sum_{i=1}^r \ln\left(1-\frac{j+i-1}{n}\right).
\end{align}

In line 3.10, we have used the fact that for $x\in\mathbb{N}$, $\Gamma(x)=\prod_{i=1}^{x-1}i.$ In line 3.12, we 
introduce the change of indices $k=(n-j+1)-i$, which ranges from 1 when $i=n-j$ to $r$ when $i=n-j-r+1$.

This lets us simplify $A_j$ into the form:
$$A_j=\left(\sum_{i=1}^r \ln\left(1-\frac{j+i-1}{n}\right)\right)-n\int_0^{(r+s)/n}\ln(1-x)\mathrm{d}x+\ln\left(1-\frac{r+s}{n}\right).$$

Writing $A_j$ in this way as sum of three quantities, it is easy to see that $A_j\leq A_1$ for all $j=1,2,\dots,s$: Only the first depends on $j$, 
and as $j$ increases, $1-\frac{j+i-1}{n}$ is decreasing, so that $\ln\left(1-\frac{j+i-1}{n}\right)$ is decreasing. This allows us to to bound 
$M_{r,s,n}~+~1~\leq~\prod_{j=1}^s~e^{A_1}~=~e^{sA_1}$. Next, we claim that $-\ln(1-x)$ is an increasing convex function on $[0,1)$. To see this, 
first, we note that the first derivative, $\frac{1}{1-x}$, is strictly positive for all $x\in[0,1)$, while 
the second derivative, $\frac{-1}{(1-x)^2}$, is strictly negative. Next, recall that the graph of a convex function $h(x)$ on any interval $[a,b]$ 
lies below the graph of the secant line from $(a,f(a))$ to $(b,f(b))$. Let $l_{[a,b]}(x)=\frac{\ln(1-a)-\ln(1-b)}{b-a}(x-a)-\ln(1-a)$ 
be the function whose graph is the secant line of $-\ln(1-x)$ from $(a,-\ln(1-a))$ to $(b,-\ln(1-b))$. We then have the inequality that 
$0\leq-\ln(1-x)\leq l_{[a,b]}(x)$ for any 
$0\leq a < x < b <1$. In particular, monotonicity of integration tells us then that 
$0\leq - \int_a^b \log(1-x)\mathrm{d}x \leq \int_a^b l_{[a,b]}(x)\mathrm{d}x = \frac{b-a}{2}(-\ln(1-b)-\ln(1-a)).$ 

Setting $a=\frac{i-1}{n}$ and $b=\frac{i}{n}$, we then have that 

\noindent$-n\int_{(i-1)/n}^{i/n}\log(1-x)\mathrm{d}x \leq \frac{1}{2}\left(-\log(1-\frac{i}{n})-\log(1-\frac{i-1}{n})\right)$. 
Which we can write as 

\noindent$\frac{1}{2}\log(1-\frac{i}{n})\leq n\int_{(i-1)/n}^{i/n}\log(1-x)\mathrm{d}x - \frac{1}{2}\log(1-\frac{i-1}{n})$. We now have the tools 
to bound $A_1$ nicely:
\begin{align}
A_1&=\left(\sum_{i=1}^r \ln\left(1-\frac{i}{n}\right)\right)-n\int_0^{(r+s)/n}\ln(1-x)\mathrm{d}x+\ln\left(1-\frac{r+s}{n}\right)\\
&=\left(2\sum_{i=1}^r \frac{1}{2}\ln\left(1-\frac{i}{n}\right)\right)-n\int_0^{(r+s)/n}\ln(1-x)\mathrm{d}x+\ln\left(1-\frac{r+s}{n}\right)\\
&\leq \left(\sum_{i=1}^r \frac{1}{2} \ln\left(1-\frac{i}{n}\right)\right) + \left(\sum_{i=1}^r n\int_{(i-1)/n}^{i/n}\ln(1-x)\mathrm{d}x - \frac{1}{2}\ln\left(1-\frac{i-1}{n}\right)\right)\\
&\;\;\;\;\;\;-n\int_0^{(r+s)/n}\ln(1-x)\mathrm{d}x+\ln\left(1-\frac{r+s}{n}\right)\\
&= \left(\sum_{i=1}^r \frac{1}{2} \ln\left(1-\frac{i}{n}\right)- \frac{1}{2}\ln\left(1-\frac{i-1}{n}\right)\right) + \left(\sum_{i=1}^r n\int_{(i-1)/n}^{i/n}\ln(1-x)\mathrm{d}x\right)\\
&\;\;\;\;\;\;-n\int_0^{(r+s)/n}\ln(1-x)\mathrm{d}x+\ln\left(1-\frac{r+s}{n}\right)\end{align}\begin{align}
&=\frac{1}{2}\ln\left(1-\frac{r}{n}\right)-n\int_{r/n}^{(r+s)/n}\ln(1-x)\mathrm{d}x+\ln\left(1-\frac{r+s}{n}\right)\\
&\leq \frac{1}{2}\ln\left(1-\frac{r}{n}\right) -\frac{s+1}{2}\left(\ln\left(1-\frac{r}{n}\right)+\ln\left(1-\frac{r+s}{n}\right)\right)+\ln\left(1-\frac{r+s}{n}\right)\\
&=-\frac{s}{2}\ln\left(1-\frac{r}{n}\right)-\frac{s-1}{2}\ln\left(1-\frac{r+s}{n}\right)\\
&\leq -\left(\frac{s}{2}+\frac{s-1}{2}\right)\ln\left(1-\frac{r+s}{n}\right)\\
&\leq -s\ln\left(1-\frac{r+s}{n}\right)
\end{align}

In lines 75-76, we have applied the bound we obtained form the convexity argument to 
one of the sums of $\frac{1}{2}\ln\left(1-\frac{i}{n}\right)$. In lines 77-78, we combine the sums of the logarithms, 
in preparation to evaluate the single telescoping sum in lines 79. In lines 80, we use again the convexity argument 
to bound the integral by the sum of two logarithms before collecting terms in 81. In line 82, we use the fact that 
$-\ln(1-x)$ is increasing. Finally, in line 83, since $-\ln\left(1-\frac{r+s}{n}\right)>0$, we use the 
slightly simpler upper bound for $\left(s-\frac{1}{2}\right)$.

We then have that $M_{r,s,n}+1\leq e^{-s^2 \ln(1-(r+s)/n)} = \left(1-\frac{r+s}{n}\right)^{-s^2}$. 
Hence, we have that $\delta_{r,s,n}\leq 2 \left( \left(1-\frac{r+s}{n}\right)^{-s^2}-1\right)$. To finish the 
proof, in the case that $r\leq s$, we repeat these arguments with their roles reversed. This 
brings us to the promised form:

$\delta_{r,s,n}\leq 2 \left( \left(1-\frac{r+s}{n}\right)^{-(\min(r,s)^2)}-1\right)=B(r,s;n).$
\end{proof}
\end{appendices}

\end{document}